\documentclass{amsart}

\usepackage{amssymb}

\usepackage{graphicx}
\usepackage{algorithm,algorithmic}

\usepackage[latin1]{inputenc}
\usepackage{amssymb,amsmath,mathtools}
\usepackage{verbatim}
\usepackage{color}
\usepackage{algorithm,algorithmic}
\usepackage{enumerate}
\usepackage{url}

\input{xypic}

\newtheorem{theorem}{Theorem}[section]
\newtheorem{lemma}[theorem]{Lemma}
\newtheorem{proposition}[theorem]{Proposition}

\theoremstyle{definition}
\newtheorem{definition}[theorem]{Definition}

\newtheorem{claim}{Claim}

\newtheorem*{HC}{(H)}
\newtheorem*{HSC}{(H')}
\newtheorem*{HDRZC}{(DRZ)}
\newtheorem*{HSDRZC}{(DRZ')}

\theoremstyle{remark}
\newtheorem{remark}[theorem]{Remark}

\numberwithin{equation}{section}

\newcommand\Z{\ensuremath{\mathbb Z}}\newcommand\A{\ensuremath{\mathbb A}}
\newcommand\Q{\ensuremath{\mathbb Q}}\newcommand\R{\ensuremath{\mathbb R}}
\newcommand\C{\ensuremath{\mathbb C}}\newcommand\F{\ensuremath{\mathbb F}}
\newcommand\Qb{{\overline\Q}}
\newcommand\cM{\ensuremath{\mathcal M}}

\newcommand\CM{\ensuremath{\operatorname{CM}}}

\newcommand\disc{\operatorname{disc}}

\newcommand\Div{\operatorname{Div}}
\newcommand\End{\operatorname{End}}

\newcommand\Gal{\operatorname{Gal}}
\newcommand\GL{\operatorname{GL}}
\newcommand\Hom{\operatorname{Hom}}

\newcommand\Ind{\operatorname{Ind}}

\newcommand\M{\operatorname{M}}

\newcommand\Nm{\operatorname{Nm}}

\newcommand\Pic{\operatorname{Pic}}
\newcommand\PGL{\operatorname{PGL}}

\newcommand\SL{\operatorname{SL}}

\newcommand\Tr{\operatorname{Tr}}

\newcommand\Jac{\operatorname{Jac}}
\renewcommand{\cH}{\mathcal{H}}
\newcommand{\fp}{{\mathfrak{p}}}

\newcommand{\B}{{\mathcal{B}}}

\newcommand{\oR}{{\mathcal{R}}}
\newcommand{\cE}{{\mathcal{E}}}
\newcommand{\cV}{{\mathcal{V}}}
\newcommand{\cT}{{\mathcal{T}}}
\newcommand{\fd}{{\mathfrak{d}}}
\newcommand{\cU}{{\mathcal{U}}}

\def\cO{{\mathcal O}}

\newcommand{\smtx}[4]{\left(\begin{smallmatrix}#1&#2\\#3&#4\end{smallmatrix}\right)}

\newcommand{\colvec}[2]{\left(\begin{matrix}#1\\#2\end{matrix}\right)}

\def\M{\operatorname{M}}

\def\p{\mathfrak p}\def\P{\mathbb P}

\newcommand{\comp}{\begin{picture}(6,5)(-3,-2)\put(0,1){\circle{2}} \end{picture}}\def\circ{\comp}
\def\fN{\mathfrak N}

\def\rec{\operatorname{rec}}

\newcommand{\ra}{\rightarrow}

\newcommand{\lra}{\longrightarrow}

\newcommand{\cerednikdrinfeld}{\v{C}erednik--Drinfel'd}

\def\Xint#1{\mathchoice
{\XXint\displaystyle\textstyle{#1}}%
{\XXint\textstyle\scriptstyle{#1}}%
{\XXint\scriptstyle\scriptscriptstyle{#1}}%
{\XXint\scriptscriptstyle\scriptscriptstyle{#1}}%
\!\int}
\def\XXint#1#2#3{{\setbox0=\hbox{$#1{#2#3}{\int}$}
\vcenter{\hbox{$#2#3$}}\kern-.5\wd0}}

\begin{document}

\title[A $p$-adic construction of ATR points on $\Q$-curves]{A $p$-adic construction of ATR points on $\Q$-curves}

\author{Xavier Guitart}
\address{ Institut f\"ur Experimentelle Mathematik, Essen}
\curraddr{}
\email{xevi.guitart@gmail.com}

\author{Marc Masdeu}
\address{University of Warwick, Warwick}
\curraddr{}
\email{m.masdeu@warwick.ac.uk}


\date{\today}

\dedicatory{}

\begin{abstract}
In this note we consider certain elliptic curves defined over real quadratic fields isogenous to their Galois conjugate. We give a construction of algebraic points on these curves defined over almost totally real number fields. The main ingredient is the system of Heegner points arising from Shimura curve uniformizations.  In addition, we provide an explicit $p$-adic analytic formula which allows for the effective, algorithmic calculation of such points.
\end{abstract}

\maketitle

\section{Introduction}
At the beginning of the 2000's Darmon introduced two constructions of local points on modular elliptic curves over number fields: the Stark--Heegner points \cite{Da1} and the ATR points \cite[Chapter 8]{Da2}. Both types of points are expected to be algebraic and to behave in many aspects as the more classical Heegner points. Although the two constructions bear some formal resemblances, a crucial difference lies in the nature of the local field involved: while the former is $p$-adic, the later is archimedean. In order to explain the importance of this distinction, let us briefly recall the constructions and some of the features that are currently known about them.

Let $E$ be an elliptic curve defined over $\Q$ of conductor $N$, and let $K$ be a real quadratic field such that the sign of the functional equation of $L(E/K,s)$ is $-1$. Let $p$ be a prime that exactly divides $N$ and which is inert in $K$. Under an additional Heegner-type hypothesis, Stark--Heegner points in $E(K_p)$ are constructed in \cite{Da1} by means of certain $p$-adic line integrals, and they are conjectured to be global and defined over narrow ring class fields of $K$. This construction was generalized by Greenberg \cite{Gr} to the much broader setting in which $E$ is defined over a totally real number field $F$ of narrow class number $1$, the extension $K/F$ is any non-CM quadratic extension in which some prime divisor of the conductor of $E$ is inert, and $L(E/K,s)$ has sign $-1$.

There is extensive numerical evidence in support of the rationality of such $p$-adic points (cf. \cite{darmon-green}, \cite{darmon-pollack}, \cite{GM-elt}, \cite{GM-quat}), but the actual proof in general seems to be still far out of reach. In spite of this, in some very special cases Stark--Heegner points are known to be global. In these particular settings they coexist with Heegner points, and they can actually be seen to be related to them (cf. \cite{BD-genus}, \cite{LV}, \cite{GrSeSh}). The $p$-adic nature of the points seems to play a key role in these arguments, by means of the connection between the formal logarithm of the Stark--Heegner points and the special values of suitable $p$-adic $L$-functions (see also \cite{BDP}, \cite{DR1}, \cite{DR2}, and \cite{BDR}).

The archimedean counterparts to these points, as introduced in \cite[Chapter 8]{Da2} and later generalized by Gartner \cite{Ga-art}, seem to be even more mysterious. The simplest and original setting is that of an elliptic curve $E$ defined over a totally real number field $F$, and $M/F$ a quadratic \emph{Almost Totally Real} (ATR) extension (i.e., $M$ is has exactly one complex place). In this case the points are constructed by means of complex integrals and thus they lie in $E(\C)$. They are also expected to be global, and there is some numerical evidence of it \cite{DL}, \cite{GM-elt}. 

However, in the archimedean constructions it is not (the logarithm of) the points which is expected to be related to special values of complex $L$-functions, but their heights, very much in the spirit of Gross--Zagier formulas. It is this crucial difference with the $p$-adic case what seems to prevent the success of any attempt of showing their rationality, even in the very particular instances in which they coexist with Heegner points. It could be arguably said that complex ATR points are much more difficult to handle than their $p$-adic counterparts. Thus, even in the simplest situations in which one wants to compare them with Heegner points in order to show their rationality, it is desiderable to have $p$-adic constructions of such points at one's disposals.

In light of the above discussion, the goal of the present paper is to present a $p$-adic construction of algebraic points defined over ATR fields. To be more precise, we consider a real quadratic field $F$ and a non-CM elliptic curve $E/F$ that is $F$-isogenous to its Galois conjugate (this is sometimes referred to as a $\Q$-curve in the literature). Suppose that $M/F$ is a quadratic ATR extension such that the sign of the functional equation of $L(E/M,s)$ is $-1$. We describe a $p$-adic construction of algebraic points in $E(M)$, which are manufactured by means of suitable Heegner points in a certain Shimura curve parametrizing $E$.

The points that we construct are algebraic (for they essentially come from Heegner points in certain modular abelian varieties) and given in terms of $p$-adic line integrals. Observe that in this set up one can also consider $p$-adic Stark--Heegner points, e.g., the ones constructed by Greenberg \cite{Gr}, and it would be very interesting to investigate the possible relationship between these two types of points. Of course, it would also be of interest to compare them with the ATR points constructed by Darmon and G\"artner, although as explained before in this case only information about the heights of the points seems to be directly available from the involved $L$-functions. 

The fact that our construction is given in terms of $p$-adic line integrals also has another consequence, which constitutes in fact one of the remarkable features of the construction: it gives rise to a completely explicit and efficient algorithm for computing the points.

Our construction is inspired by the work of Darmon--Rotger--Zhao \cite{DRZ}. Since it builds on this work, the next section is devoted to recalling the points introduced in \cite{DRZ}, as well as to giving an overview of the rest of the paper.

\subsection*{Acknowledgments} We thank Victor Rotger for suggesting the problem and Jordi Quer for providing the equation of the $\Q$-curve used in  \S\ref{section: example}. We also thank the anonymous referee, whose many detailed and helpful comments allowed us to significantly improve both the content and the clarity of this article. Guitart was financially supported by SFB/TR 45.

\section{Background and outline of the construction} Our construction can be seen as a generalization of that of~\cite{DRZ}. In order to put it in context, it is illustrative to examine first the case of elliptic curves over $\Q$. So let us (temporarily) denote by $E$ an elliptic curve over $\Q$ of conductor $N$. The Modularity Theorem \cite{Wi},\cite{TW}, \cite{BCDT} provides a non-constant map
\begin{align}\label{eq: classical uniformization}
 \pi_E\colon  X_0(N)\lra E,
\end{align}
where $X_0(N)$ denotes the modular curve parametrizing cyclic isogenies $C\ra C'$ of degree $N$. This moduli interpretation endows $X_0(N)$ with a canonical set of algebraic points known as CM  or Heegner points which give rise, when projected under $\pi_E$, to a systematic construction of algebraic points on $E$.

To be more precise, suppose that $K$ is a quadratic imaginary field  and $\cO\subset K$ is an order of discriminant coprime with $N$. In addition, suppose that $K$ satisfies the \emph{Heegner condition}:
\begin{HC}
 All the primes dividing $N$ are split in $K$. 
\end{HC}
  Under this assumption, there exist elliptic curves $C$ and $C'$ with complex multiplication by $\cO$, together with  a cyclic isogeny $C\ra C'$ of degree $N$. The theory of complex multiplication implies that the point in $X_0(N)$ corresponding to $C\ra C'$ is, in fact, algebraic and defined over the ring class field of $\cO$.

Moreover, the corresponding Heegner point on $E$  can be computed by means of the complex uniformization derived from \eqref{eq: classical uniformization} which, in view of the identifications $X_0(N)(\C)\simeq \Gamma_0(N)\backslash (\cH \cup \P^1(\Q))$ and $E(\C)\simeq \C/\Lambda_E$, is of the form
\begin{align}
\pi_E\colon\Gamma_0(N)\backslash (\cH \cup \P^1(\Q)) \lra \C/\Lambda_E.
\end{align}
The formula for computing the Heegner point corresponding to $C\ra C'$ is then
\begin{align}\label{eq: hp}
  \Phi_{\text{W}}\left( \int_\tau^{i\infty}2\pi i f_E(z)dz \right),
\end{align}
where $f_E(z)=\sum_{n\geq 1} a_n e^{2\pi i nz}$ denotes the weight two newform for $\Gamma_0(N)$ whose $L$-function equals that of $E$, the map $\Phi_W\colon \C/\Lambda_E \ra E(\C)$ is the Weierstrass uniformization, and $\tau\in \cH\cap K$ is such that $C\simeq \C/\Z+\tau\Z$ and $C'\simeq \C/\Z+N\tau\Z$. 

This type of Heegner points are one of the main ingredients intervening in the proof of the Birch and Swinnerton-Dyer conjecture for curves over $\Q$ of analytic rank at most $1$ \cite{GZ},\cite{Ko}. Moreover, and perhaps more relevant to the purpose of the present note, the formula \eqref{eq: hp} is completely explicit and computable, as the Fourier coefficients $a_n$ can be obtained by counting points on the several reductions of $E\pmod p$. In other words, \eqref{eq: hp} provides with an effective algorithm for computing points on $E$ over abelian extensions of $K$, which turn out to be of infinite order whenever the analytic rank is $1$. See, e.g., \cite{El} for a discussion of this method and examples of computations.

Suppose now that $K$ does not satisfy the Heegner condition, and factor $N$ as $N=N^+ N^-$, where $N^+$ contains the primes that split in $K$ and $N^-$ those that are inert (we assume that the discriminant of $K$ is coprime to $N$). In this case there is a generalization of the above Heegner point construction, which works under the less restrictive \emph{Heegner--Shimura condition}:
\begin{HSC}
   $N^-$ is squarefree and the product of an even number of primes.
\end{HSC}
In this set up, when $N^->1$ the map \eqref{eq: classical uniformization} is replaced by a uniformization of the form
\begin{align}\label{eq: shimura uniformization}
  \pi_E^{N^-}\colon X_0(N^+,N^-)\lra E,
\end{align}
where $X_0(N^+,N^-)$ is the Shimura curve of level $N^+$ associated to the indefinite quaternion algebra $B/\Q$ of discriminant $N^-$. The moduli interpretation of $X_0(N^+,N^-)$, combined with the theory of complex multiplication, can also be used to construct Heegner points on $E$ that are defined over ring class fields of orders $\cO\subset K$ of conductor coprime to $N$. 

There is also an analogue of formula \eqref{eq: hp}, but in this case it seems to be much more difficult to compute in practice. Indeed, one needs to integrate modular forms associated to $B$ and, since $B$ is division, the Shimura curve $X_0(N^+,N^-)$ has no cusps. Therefore the corresponding modular forms do not admit Fourier expansions, which are the crucial tool that allow for the explicit calculation of \eqref{eq: hp}.  Elkies developed methods for performing such computations under some additional hypothesis \cite{El-SC}. More recently Voight--Willis \cite{VW} using Taylor expansions  and Nelson \cite{nelson} using the Shimizu lift have been able to compute some of these CM points.

In a different direction, there is an alternative method that allows for the numerical calculation of Heegner points associated to quaternion division algebras. The key idea is to fix a prime $p\mid N^-$ and to use the rigid analytic $p$-adic uniformization derived from \eqref{eq: shimura uniformization}, instead of the complex one. The \cerednikdrinfeld{} theorem provides a model for $X_0(N^+,N^-)$ as the quotient of the $p$-adic upper half plane $\cH_p$ by $\Gamma$, a certain subgroup in a definite quaternion algebra. Bertolini and Darmon \cite{BD-Heegner-points}, building on previous work of Gross \cite{Gross}, give an explicit formula for the uniformization map
\begin{align*}
  \Gamma\backslash \cH_p \lra E(\C_p)
\end{align*}
in terms of the so-called multiplicative $p$-adic line integrals of rigid analytic modular forms for $\Gamma$ (see also \cite{MOK} for a generalization to curves over totally real fields). Such integrals can be very efficiently computed, thanks to the methods of M. Greenberg \cite{Gr2} (which adapt Pollack--Stevens' overconvergent modular symbols technique \cite{PS}) and to the explicit algorithms provided by Franc--Masdeu \cite{FM}.

Let us now return to the setting that we consider in the present note. Namely, $F$ is a real quadratic field and $E/F$ is an elliptic curve without complex multiplication that is $F$-isogenous to its Galois conjugate. As a consequence of Serre's modularity conjecture \cite{KW1,KW2} and results of Ribet \cite{Ribet}, $E$ can be parametrized by a modular curve of the form $X_1(N)$, where the integer $N$ is related to the conductor $\mathfrak{N}_E$ of $E/F$ (cf. \S \ref{subsection: Q-curves} for more details on the precise relation). This property was exploited by Darmon--Rotger--Zhao in \cite{DRZ} in order to construct certain algebraic ATR points on $E$ by means of Heegner points on $X_1(N)$ . Let us briefly explain the structure of the construction.

Consider the uniformization mentioned above
\begin{align}\label{eq: uniformization X_1}
 \pi_E\colon  X_1(N)\lra E.
\end{align}
We assume, for simplicity, that $N$ is squarefree. We remark that $\pi_E$ is defined over $F$. Let $M/F$ be a quadratic extension that has one complex and two real places (this is what is known as an Almost Totally Real (ATR) extension, because it has exactly one complex place). There is a natural quadratic imaginary field $K$ associated to $M$ as follows: if $M=F(\sqrt{\alpha})$ for some $\alpha\in F$, then $K=\Q(\sqrt{\Nm_{F/\Q}(\alpha)})$. Suppose that $K$ satisfies the following Heegner-type condition, which might be called the Heegner--Darmon--Rotger--Zhao condition:
\begin{HDRZC}
  All the primes dividing $N$ are split in $K$.
\end{HDRZC}
Under this assumption, the method presented in \cite{DRZ} uses Heegner points on $X_1(N)$ associated to orders in $K$ to construct points in $E(M)$, which are shown to be of infinite order in situations of analytic rank one. One of the salient features of this construction is that it is explicitly computable. In fact, there is a formula analogous to \eqref{eq: hp}, giving the points as integrals of certain classical modular forms for $\Gamma_1(N)$.

In the first part of the paper, which consists of Sections \ref{section: Q-curves and ATR extensions} to \ref{section: Heegner points on Q-curves}, we extend the construction of \cite{DRZ} to the situation in which $K$, whose discriminant is assumed to be coprime to $N$, satisfies the following, less restrictive, Heegner--Shimura-type condition (where, as before, we write $N=N^+ N^-$ with $N^+$ containing the primes that split in $K$ and $N^-$ containing those that remain inert):
\begin{HSDRZC}
   $N^-$ is squarefree and the product of an even number of primes.
\end{HSDRZC}
As we will see, this condition is satisfied whenever $L(E/M,s)$ has sign $-1$ (see Proposition \ref{prop: primes dividing N minus} below). In particular, it is satisfied when the analytic rank of $E/M$ is $1$.

The idea of our construction, inspired by the case of curves over $\Q$ reviewed above, consists in replacing \eqref{eq: uniformization X_1} by a uniformization of the form
\begin{align}
  \pi_E^{N^-}\colon X_1(N^+,N^-)\lra E,
\end{align}
where $X_1(N^+,N^-)$ is a suitable Shimura curve attached to an indefinite quaternion algebra $B/\Q$ of discriminant $N^-$ and level structure ``of $\Gamma_1$-type''. This main construction of ATR points in $E(M)$ is presented in \S \ref{section: Heegner points on Q-curves}, after developing some preliminary results. Namely, in \S \ref{section: Q-curves and ATR extensions} we briefly review $\Q$-curves and we prove some results in Galois theory that relate certain ring class fields of $K$ with $M$, and in \S \ref{section: CM points on Shimura curves with quadratic character} we define the CM points on the Shimura curves that will play a role in our construction and determine their field of definition.

Just as in the classical case of curves over $\Q$, the CM points in $X_1(N^+,N^-)$ (and hence the points that we construct in $E(M)$) are difficult to compute using the complex uniformization. Once again, the absence of cusps in $X_1(N^+,N^-)$ and thus the lack of Fourier coefficients makes it difficult to compute the integrals that appear in the explicit formula (cf.~\eqref{eq: pi_f(D) complex} below).

The second part of the article gives a $p$-adic version of the construction. As has been mentioned in the introduction, this might be useful in order to relate it to $p$-adic Stark-Heegner points. Another advantage of this $p$-adic construction is that it is explicitly computable. Concretely, in \S \ref{section: p-adic uniformization} we exploit the $p$-adic uniformization of $X_1(N^+,N^-)$ given by the \cerednikdrinfeld{} theorem and the explicit uniformization of Bertolini--Darmon in terms of multiplicative $p$-adic integrals which, combined with a slight generalization of the algorithms of Franc--Masdeu \cite{FM}, provides an efficient algorithm for computing algebraic ATR points in $\Q$-curves. We conclude with an explicit example of such computation in \S \ref{section: example}.

\section{$\Q$-curves and ATR extensions}\label{section: Q-curves and ATR extensions}
In this section we recall some basic facts on $\Q$-curves and their relation with classical modular forms for $\Gamma_1(N)$. We also give some preliminary results on certain Galois extensions associated to ATR fields that will be needed in the subsequent sections, as they will be related to the field of definition of the Heegner points under consideration.

\subsection{$\Q$-curves and modular forms}\label{subsection: Q-curves} 
An elliptic curve $E$ over a number field is said to be a \emph{$\Q$-curve} if it is isogenous to all of its Galois conjugates. One of the motivations for studying $\Q$-curves is that they arise as the $1$-dimensional factors over $\Qb$ of the modular abelian varieties attached to classical modular forms. More precisely, let $f = \sum_{n\geq 1}a_nq^n$ be a normalized newform of weight two for $\Gamma_1(N)$, and denote by $K_f=\Q(\{a_n\})$ the number field generated by its Fourier coefficients. Let $A_f$ be the abelian variety over $\Q$ associated to $f$ by Shimura in \cite[Theorem 7.14]{shimura-book}. The dimension of $A_f$ is equal to $[K_f\colon\Q]$ and $\End_\Q(A_f)\otimes\Q$, its algebra of endomorphisms defined over $\Q$, is isomorphic to $K_f$. In particular, $A_f$ is simple over $\Q$. However, it is not necessarily absolutely simple and, in general, it decomposes up to $\Qb$-isogeny as $A_f\sim_\Qb C^n$, for some $n\geq 1$ and some absolutely simple abelian variety $C/\Qb$ which is isogenous to all of its Galois conjugates. Therefore, if $C$ turns out to be of dimension $1$, then $C$ is a $\Q$-curve. Conversely, as a consequence of Serre's modularity conjecture and results of Ribet, any $\Q$-curve can be obtained, up to isogeny, by this construction.

We will be interested in certain $\Q$-curves defined over quadratic fields. In the next proposition we characterize them in terms of the modular construction.
\begin{proposition}\label{prop: determination of E}
Let $f\in S_2(\Gamma_1(N))$ be a non-CM newform of Nebentypus $\psi$, and let $F$ be the field associated to $\psi$ (by identifying $\psi$ with a character of $\Gal(\Qb/\Q)$ via class field theory).  Suppose that $\psi$ has order $2$ and that $[K_f\colon\Q]=2$. Then:
\begin{enumerate}
\item $F$ is real quadratic;
\item $K_f$ is imaginary;
\item $A_f$ is $F$-isogenous to $E^2$, where $E/F$ is a $\Q$-curve.
\end{enumerate}
In addition, if $N$ is odd or squarefree then the conductor $\fN_E$ of $E/F$ is generated by a rational integer, say $\fN_E=N_0\cO_F$ for some $N_0\in\Z_{\geq 0}$, and $N = N_0 \cdot N_\psi$ where $N_\psi$ stands for the conductor of $\psi$.
\end{proposition}
\begin{proof}
The field $F$ is real because $\psi$, being the Nebentypus of a modular form of even weight, is an even character. 

The Nebentypus induces the complex conjugation on the Fourier coefficients:
  \begin{align}\label{eq: inner}
    a_p = \bar a_p \psi(p),\ \text{ for almost all primes $p$ (cf., e.g., \cite[\S1]{RibetGR})}.
  \end{align}
Since $\psi$ has order $2$ it is non-trivial, which implies that $K_f$ is imaginary.

Property \eqref{eq: inner} also implies that $\psi$ is an inner twist of $f$, in the sense of \cite[\S3]{RibetTwists}. In fact, it is the only non-trivial inner twist because $[K_f\colon \Q]=2$. By \cite[Proposition 2.1]{Pep-Lario} $F$ is the smallest number field where all the endomorphisms of $A_f$ are defined. If $A_f$ were absolutely simple, then $\End_F(A_f)\otimes\Q$ would be isomorphic to a quaternion division algebra over $\Q$ by \cite[Proposition 1.3]{Pyle}. But in that case the minimal field of definition of all the endomorphisms of $A_f$ would be quadratic imaginary by \cite[Lemma 2.3 (i)]{Ro}, so this case is not possible. We conclude that $A_f$ is not absolutely simple, so that $A_f$ is $F$-isogenous to $E^2$, where $E/F$ is a $\Q$-curve.

Suppose now that $N$ is odd or squarefree. Then, by the main theorem in \cite[p. 2]{GoGu} one has that the conductor $\fN_E$ of $E/F$ is of the form $N_0\cO_F$ for some $N_0\in\Z_{\geq 0}$, and the relation $N_0 N_F = N$ is satisfied, where $N_F$ stands for the conductor of $F$. But $N_F = N_\psi$ by the conductor-discriminant formula~\cite[Theorem 3.11]{washington-book}, and this finishes the proof.
\end{proof}
From now on we will assume that $E/F$ is an elliptic curve obtained as in the above proposition. That is to say, it is the absolutely simple factor of $A_f$ for some non-CM newform $f$ whose level $N$ is odd or squarefree, its Nebentypus is of order $2$ and its field of Fourier coefficients is quadratic. In addition, in the case where $N$ is not squarefree we will also assume that $(N_0,N_\psi)=1$, since this condition will be needed in Proposition \ref{prop: primes dividing N minus}.  One can easily find many examples of modular forms satisfying these conditions, for instance by consulting the table \cite[\S4.1]{Quer}. 

\begin{remark}\label{remark: N}
  Observe that the above assumptions imply that $N_\psi$ is squarefree. This is clear if $N$ is squarefree. If $N$ is odd, it follows from the fact that the conductor of a quadratic character is squarefree away from $2$.
\end{remark}

\subsection{ATR extensions}\label{subsection: ATR}

Let $M/F$ be a quadratic almost totally real (ATR) extension of discriminant prime to $\fN_E$, the conductor of $E$, and such that the $L$-function $L(E/M,s)$ has sign $-1$.
This condition is equivalent (see, e.g., the discussion of \cite[\S3.6]{Da2}) to the set
\begin{equation}\label{eq: primes inert}
 \{\p \mid \fN_E\colon \p \text{ is inert in } M\}
\end{equation}
having even cardinality.

We have that $M=F(\sqrt{\alpha})$ for some $\alpha\in F$. We set $M'=F(\sqrt{\alpha'})$, where  $\alpha'$ stands for the Galois conjugate of $\alpha$. Then $\mathcal
M=M M'$ is the Galois closure of $M$ and its Galois group
$\Gal(\mathcal M/\Q)$ is isomorphic to $ D_{2\cdot 4}$, the dihedral group of $8$ elements. The diagram of subfields of $\cM$ is of the form
\begin{equation}\label{eq: diagram}
\xymatrix{
         &           &  \cM \ar@{-}[d]\ar@{-}[dll]\ar@{-}[drr] \ar@{-}[dl]\ar@{-}[dr]     &         &           \\
M   & M'   &      FK      &    L    &    L'     \\
  &  F\ar@{-}[ul]\ar@{-}[u]\ar@{-}[ur] & K' \ar@{-}[u] & K
\ar@{-}[ul]\ar@{-}[u]\ar@{-}[ur]\\
    & &\Q \ar@{-}[ul]\ar@{-}[u]\ar@{-}[ur] & &
}
\end{equation}
where $K=\Q(\sqrt{\alpha\alpha'})$. Observe that $K$ is a quadratic imaginary field, for $M$ is ATR and necessarily
$\alpha\alpha'=\Nm_{F/\Q}(\alpha)<0$. From now on, we will assume that the discriminant of $K$ is relatively prime to $N$. 

We will see that all the primes dividing $N_\psi$ are split in $K$ (see Lemma \ref{lemma: new}). We consider a decomposition of $N$ of the form
$N=N^+ N^-$, where
\begin{itemize}
\item $N^+=N_\psi  N_0^+$, and $N_0^+$  contains the primes $\ell \mid N_0$ such that $\ell$ splits in $K$, and
\item $N^-$  is squarefree and contains the primes $\ell\mid N_0$ such that $\ell$ is inert in $K$.
\end{itemize}

As we already mentioned in the Introduction, one of the central ideas of \cite{DRZ} is that  Heegner points on $A_f$ can be used to manufacture points on
$E(M)$. Indeed, an explicit such construction is provided in  \cite[\S4]{DRZ}, under the assumption that $\fN_E=(1)$. Such construction, in fact, is easily seen to be valid under the following slightly more general Heegner-type condition:
\begin{HDRZC}
 $N^- =1$ (i.e., all the primes dividing $N$ are split in $K$). 
\end{HDRZC}
Let us briefly review the structure of the construction in this case (we refer to \cite{DRZ} for the details). Let us (temporarily) denote by $\Gamma_0(N)$ the subgroup of $\SL_2(\Z)$ of upper triangular matrices modulo $N$, and by $\Gamma_\psi(N)$ the congruence subgroup
\begin{align*}
  \Gamma_\psi(N)=\{\smtx a b c d\in\SL_2(\Z) \colon N\mid c, \, \psi(a)=1\}\subset\Gamma_0(N).
\end{align*}
Let $X_0(N)$ (resp. $X_\psi(N)$) denote the modular curve associated to $\Gamma_0(N)$ (resp. to $\Gamma_\psi(N)$), and let $J_0(N)$ (resp. $J_\psi(N)$) denote its Jacobian. The variety $A_f/\Q$ turns out to be a quotient of $J_\psi(N)/\Q$.  Since $A_f$ is isogenous over $F$ to $E^2$, it follows
that $E$ admits a morphism (defined over $F$) from $J_\psi(N)$. Therefore we obtain a uniformization
\begin{align}\label{eq: uniformization 1}
  J_\psi(N) \lra E
\end{align}
which is defined over $F$.

On the other hand, the inclusion $\Gamma_\psi(N)\subset \Gamma_0(N)$ induces a degree $2$ map
$ X_\psi(N)\ra X_0(N)$, 
and the Heegner points in $X_\psi(N)$ are the preimages of the Heegner points in $X_0(N)$. Denote by $M_0(N)\subset \M_2(\Z)$ the set of matrices which are upper triangular modulo $N$. An embedding $\varphi\colon K\hookrightarrow M_2(\Q)$ is said to be of conductor $c$ and level $N$ if $\varphi^{-1}(M_0(N))$ is equal to $\cO_c$, the order of conductor $c$. The Heegner points in $X_0(N)$ associated to $\cO_c$ are in one to one correspondence with the optimal embeddings of level $N$ and conductor $c$. They are defined over its ring class field $H_c$, so that their preimages in $X_\psi(N)$ are defined over a certain quadratic extension $L_c$ of $H_c$ (see \S\ref{subsection: CM points} for the details). This gives rise to Heegner points in $J_\psi(N)$ defined over $L_c$.

 One of the results proved in \cite{DRZ} is that, for suitable choices of $c$, the extension $L_c$ contains $L$ (cf. \S \ref{subsec: Extensions of L} for a generalization of this result). By taking the trace from $L_c$ down to $L$ one obtains a point in $J_\psi(N)(L)$. Summing it with its conjugate by an appropriate element in
$\Gal(\cM/\Q)$ produces a point on $J_\psi(N)(M)$. Finally, projecting  to $E$ via \eqref{eq: uniformization 1} yields the point on $E(M)$.

The  reason why the construction outlined above only works under the hypothesis that $N^-=1$ is that, otherwise, there
do not exist
optimal embeddings $\varphi\colon K\hookrightarrow M_2(\Q)$ of conductor $c$ and level $N$. That is to say, there are no Heegner points in $X_0(N)$ defined over ring class fields of $K$. 

The main goal of the present article is to provide a construction of Heegner points on $E(M)$ in the case $N^->1$. For
that purpose, and similarly to the classical case of Heegner points on curves over $\Q$, we need to consider
Heegner points coming from  Shimura curves attached to division quaternion algebras. In the next section we
introduce the Shimura curves that will play the role of $X_\psi(N)$ in our construction, and we discuss Heegner points
on them. 

Before that, we state some Galois properties of the fields in Diagram \eqref{eq: diagram} and about certain number
fields $L_c$, attached to orders in $K$ of conductor $c$ that will be the fields of definition of Heegner points. We
also introduce some more notation that will be in force for the rest of the article.

\subsection{Galois properties and the number of primes dividing $N^-$} In this subsection we study those properties of the field diagram \eqref{eq: diagram} that are needed later. Let \begin{align*}\chi_M,\chi_M'\colon G_F\lra \{\pm 1\}\end{align*} denote the quadratic characters of $G_F=\Gal(\Qb/F)$ cutting out the extensions $M$ and $M'$, respectively. Observe that we
can, and often do, view them as characters on the ideles $\A_F^\times$. Similarly we define the characters
\[
 \chi_L,\chi_L'\colon G_K\lra \{\pm 1\},
\]
and view them as characters of $\A_K^\times$. 
\begin{remark}\label{rk: M and M'}
Observe that $M=F(\sqrt{\alpha})$ by construction,  and $M'=F(\sqrt{\alpha'})$, where $\alpha'$ denotes the $\Gal(F/\Q)$-conjugate of $\alpha$. In particular, if $p$ is a prime that splits in $F$, say as $p\cO_F=\fp\fp'$, then $ \chi_M(\fp)=\chi_M'(\fp')$, for the splitting behavior of $\fp$ in $M/F$ is the same as that of $\fp'$ in $M'/F$. A similar observation applies for $L$.
\end{remark}

We also denote by $\varepsilon_F$ and $\varepsilon_K$ the quadratic characters on $\A_\Q^\times $ corresponding to $F$
and $K$, and by
\[
\Nm_{F/\Q}\colon \A_F^\times \lra \A_{\Q}^\times,\ \ \Nm_{K/\Q}\colon \A_K^\times \lra \A_{\Q}^\times 
\]
the norms on the ideles. Observe that, as remarked above, $F$ is the field cut out by $\psi$. This means that, in fact,  $\varepsilon_F=\psi$. 

We will make use of the following properties of Diagram
\eqref{eq: diagram}, which are given in Proposition 3.2 of \cite{DRZ}.
\begin{lemma}\label{lemma: diagram properties}
 \begin{enumerate}
  \item\label{item1} $\chi_M\cdot \chi_M'=\varepsilon_K\circ \Nm_{F/\Q}$ and $\chi_L\cdot\chi_L'=\varepsilon_F\circ \Nm_{K/\Q}$.
\item\label{item2} The restriction of $\chi_M$ and $\chi_M'$ to $\A_{\Q}^\times$ is $\varepsilon_K$, and the restriction of
$\chi_L$ and $\chi_L'$ to $\A_{\Q}^\times$ is $\varepsilon_F$.
\item $\operatorname{Ind}_F^\Q \chi_M = \operatorname{Ind}_K^\Q\chi_L.$
 \end{enumerate}
\end{lemma}

Let $\mathfrak d_{L/K}$ denote the discriminant of the extension $L/K$, which by the conductor-discriminant formula is the conductor of $\chi_L$.

\begin{lemma}\label{lemma: new}
  \begin{enumerate}
  \item   There exists an ideal $\fN_{\psi,L}\subset \cO_K$ of norm $N_\psi$ canonically attached to $(\psi,L)$. In particular, all primes dividing $N_\psi$ are split in $K$.
\item One may write $\mathfrak d_{L/K}=c \cdot \fN_{\psi,L},$
were $c$ belongs to $ \Z$ and is coprime to $N$.
  \end{enumerate}
\end{lemma}
\begin{proof}
  From the equality $\operatorname{Ind}_F^\Q\chi_M =\operatorname{Ind}_K^\Q\chi_L$, using the formula for the conductor of induced representations and the conductor-discriminant formula, we obtain
  \begin{align}\label{eq: relation between discriminants}
    N_\psi \cdot \Nm_{F/\Q}(\fd_{M/F})=\disc(K)\cdot \Nm_{K/\Q}(\fd_{L/K}).
  \end{align}
From this it follows that $N_\psi\mid \Nm_{K/\Q}(\fd_{L/K})$, since $N$ is coprime to $\disc(K)$. If we write 
\begin{align*}
  \fd_{L/K} = \prod_{\substack{
   \mathfrak p \mid\fd_{L/K} \\
   \mathfrak p\nmid N_\psi
  }} \mathfrak{p}^{e_\mathfrak{p}}   \cdot  \prod_{\substack{
   \mathfrak p \mid\fd_{L/K} \\
   \mathfrak p\mid N_\psi
  }} \mathfrak{p}^{e_\mathfrak{p}}
\end{align*}
then we have (by setting $p_{\mathfrak p}=\mathfrak p\cap\Z$):
\begin{align*}
 \Nm_{K/\Q} (\fd_{L/K}) = \prod_{\substack{
   \mathfrak p \mid\fd_{L/K} \\
   \mathfrak p\nmid N_\psi
  }} p_\fp^{e_\mathfrak{p}'}   \cdot  \prod_{\substack{
   \mathfrak p \mid\fd_{L/K} \\
   \mathfrak p\mid N_\psi
  }} p_\fp^{e_\mathfrak{p}'},\ \  \text{ and }\ \  N_\psi = \prod_{\substack{
   \mathfrak p \mid\fd_{L/K} \\
   \mathfrak p\mid N_\psi
  }} \mathfrak{p}^{e_\mathfrak{p}'}.
\end{align*}
Thanks to our running assumption that $(N,\disc(K))=1$, for every $\mathfrak p\mid N_\psi$ we either have that $\Nm_{K/\Q}(\mathfrak p) = p_\mathfrak{p}$ when $p_\mathfrak{p}$ is split (in which case $e_{\mathfrak p}'= e_{\mathfrak{p}}$), or that $\Nm_{K/\Q}(\mathfrak{p})=p_{\mathfrak{p}}^2$ when $p_\mathfrak{p}$ is inert (in which case $e_{\mathfrak p}'= 2e_{\mathfrak{p}}$). Since $N_\psi$ is squarefree (see Remark \ref{remark: N}), it must be $e_\mathfrak{p}'=1$ so that $e_\mathfrak{p}=1$ and $p_\mathfrak{p}$ is split. Therefore
\begin{align*}
  N_\psi = \Nm_{K/\Q} ( \prod_{\substack{
   \mathfrak p \mid\fd_{L/K} \\
   \mathfrak p\mid N_\psi
  }} \mathfrak{p}),
\end{align*}
which proves the first assertion by putting $\fN_{\psi,L}= \prod_{\mathfrak{p}\mid\fd_{L/K},\mathfrak{p}\mid N_\psi}\mathfrak{p}$.

As for the second part of the lemma, we consider primes $p$ dividing $\Nm_{K/\Q}(\fd_{L/K})$. Suppose first that $p\mid N_\psi$. By the first part $p$ splits in $K$, say as $p\cO_K=\fp\bar\fp$. Then $\cO_{K,\fp}\simeq \Z_p$, and by part \eqref{item2} of Lemma \ref{lemma: diagram properties} the  composition 
  \begin{align*}
    \Z_p^\times\lra \cO_{K,\fp}^\times \times \cO_{K,\bar\fp}^\times \stackrel{\chi_{L,\fp}\cdot \chi_{L,\bar\fp}}{\lra} \{\pm 1\}
  \end{align*}
is equal to $\psi_p$ (the local component of $\psi$ at $p$), which is non trivial because $p\mid N_\psi$. But since $\chi_{L,\fp}$, $\chi_{L,\bar\fp}$ are quadratic characters, then necessarily exactly one them is trivial on $\F_p^\times\simeq \cO_{K,\fp}^\times/(1+\fp)\simeq \cO_{K,\bar\fp}^\times/(1+\bar\fp)$, say $\chi_{L,\bar\fp}$. Then $\fp$ divides exactly the conductor of $\chi_L$ (which is equal to $\fd_{L/K}$), and $\bar\fp$ does not divide it.

Now suppose that $p\nmid N_\psi$. That is to say, $\psi_p$ is trivial on $\Z_p^\times$. Let $\fp\mid p$ be a prime in $K$ such that $\fp^e$ divides exactly the conductor of $\chi_L$. The composition 
  \begin{align*}
    \Z_p^\times\lra \cO_{K,\fp}^\times \stackrel{\chi_{L,\fp}}{\lra} \{\pm 1\}
  \end{align*}
is equal to $\psi_p$, which is trivial on $\Z_p^\times$. If $p$ was ramified in $K$, then the above map would restrict to
\begin{align*}
      \Z_p^\times/(1+p^e\Z_p)\stackrel{\simeq}{\lra} \cO_{K,\fp}^\times/(1+\fp^e\cO_{K,\fp}) \stackrel{\chi_{L,\fp}}{\lra} \{\pm 1\},
\end{align*}
contradicting the non-triviality of $\chi_{L,\fp}$ restricted to $\cO_{K,\fp}^\times/(1+\fp^e\cO_{K,\fp})$. Hence we see that $p$ cannot ramify. If $p$ is inert in $K$ there is nothing to prove, because $\fp^e=p^e\cO_K$ is already a rational ideal. If $p$ splits in $K$, say as $p\cO_K=\fp\bar\fp$, then $\chi_{L,\bar\fp}\simeq \chi_{L,\fp}^{-1}$ on $\Z_p^\times$, implying that $\bar\fp^e$  exactly divides the conductor of $\chi_L$, because $\cO_{K,\bar\fp}\simeq\Z_p\simeq \cO_{K,\fp}$. 

We have seen that $\fd_{L/K}=c\cdot \fN_{\psi,L}$ with $c\in\Z$. It remains to prove that $c$ is coprime to $N$. Recall that $N = N_\psi N_0$, where $N_0$ is a generator of the conductor $\fN_E$ of $E$. From what we have seen in the proof so far, it is clear that $(c,N_\psi)=1$. Recall also our running assumption that $\fN_E$ is coprime to $\fd_{M/K}$ (cf. \S\ref{subsection: ATR}), which implies that $(N_0,\Nm_{F/\Q}(\fd_{M/K}))=1$. From \eqref{eq: relation between discriminants} we see that $c\mid \Nm_{F/\Q}(\fd_{M/F})$, and therefore $(c,N_0)=1$.

\end{proof}
The Heegner points that we will use in our construction arise from a Shimura curve associated to an indefinite algebra of discriminant $N^-$. Therefore, the following result is key to our purposes.
\begin{proposition}\label{prop: primes dividing N minus}
 The number of primes dividing $N^-$ is even.
\end{proposition}
\begin{proof} 
Recall that $\fN_E=N_0 \cO_F$ and that the set 
\begin{align}\label{eq: primes inert 2}
 \{\p \mid \fN_E\colon \p \text{ is inert in } M\}
\end{align}
has even cardinality thanks
to our running assumption that $L(E/M,s)$ has sign $-1$. Every prime in the set \eqref{eq: primes inert 2} is above a prime $p\mid N_0$. Thus, in order to prove the proposition it is enough to prove the  following claims.
\begin{claim}
  Every prime $p\mid N_0^+$ gives rise to either zero or two primes in  \eqref{eq: primes inert 2}.
\end{claim}
\begin{claim}
  Every prime $p\mid N_0^-=N^-$ gives rise to exactly one prime in  \eqref{eq: primes inert 2}.
\end{claim}
\subsubsection*{Proof of Claim 1}Let $p$ be a prime dividing $N_0^+$. Namely, $p$ is a prime divisor of $N_0$ that splits in $K$. Observe that $p$ can not ramify in $F$ because of our assumption that $(N_\psi,N_0)=1$. If $p$ splits in $F$,
say $p \cO_F=\fp\fp'$, by Remark \ref{rk: M and M'} and part (1) of Lemma \ref{lemma: diagram properties} we see that
\[
 \chi_M(\fp) \cdot \chi_M(\fp')=\chi_M(\fp) \cdot \chi_M'(\fp)=\varepsilon_K(\operatorname{Nm}_\Q^F(\fp))=\varepsilon_K(p)=1,
\]
so that either both $\fp$ and $\fp'$ are inert in $M$, or both are split. In other words, either $\fp$ and $\fp '$ belong to \eqref{eq: primes inert 2}, or none of them does. 

If $p$ remains inert in $F$, by part (2) of Lemma \ref{lemma: diagram properties} we have that
\[
 \chi_M(p \cO_F)=\varepsilon_K(p)=1,
\]
which means that $p \cO_F$ is split in $M$, so that it does not belong to \eqref{eq: primes inert 2}.

\subsubsection*{Proof of Claim 2} Let  $p$ be a prime dividing $N_0^-= N^-$.  Again there are two possibilities. 
\begin{enumerate}
\item If $p$ is split in $F$, say $p\cO_F=\fp\fp'$, then
by part (1) of Lemma \ref{lemma: diagram properties} we have that
\[
\chi_M(\fp) \chi_M(\fp')= \chi_M(\fp) \chi_M'(\fp)=\varepsilon_K(\operatorname{Nm}_\Q^F(\fp))=\varepsilon_K(p)=-1,
\]
so exactly one of the primes above $p$ is inert in $M$ and therefore belongs to \eqref{eq: primes inert 2}.
\item   If $p$ is inert in $F$, then by part (2) of Lemma \ref{lemma:
diagram properties} we see that
\[
  \chi_M(p \cO_F)=\varepsilon_K(p)=-1,
\]
and so $p\cO_F$ is inert in $M$. 
\end{enumerate}

\end{proof}

\subsection{The field $L_c$}\label{subsec: Extensions of L} The aim of this subsection is to define an extension $L_c$ of $L$, associated to $\psi$ and to the order of conductor $c$ in $K$.  It will turn out to be the field of definition of the Heegner points that we will consider in Section \ref{section: CM points on Shimura curves with
quadratic character}. 

Recall that from Lemma  \ref{lemma: new} the discriminant
of $L/K$ factors as 
\[
 \mathfrak d_{L/K}=c \cdot \fN_{\psi,L},
\]
where $c$ is a rational integer with $(c,N)=1$ and $\fN_{\psi,L}$ is an ideal in $K$ of norm $N_\psi$. Let $N^+=N_\psi N_0^+$ and let  $\fN^+$ be an ideal of $K$ of norm $N^+$, such that $\fN_{\psi,L}\mid \fN^+$. We remark that $\fN_{\psi,L}$ is determined by $(\psi,L)$, while there is some freedom in the choice of $\fN^+$. We denote by $\bar\fN^+$ its complex conjugate.

Let $H_c/K$ be the ring class field of $K$ of conductor $c$. Denote by $\A_K$ the adeles of $K$, and by $\hat\cO_K=\prod_\fp \cO_{K,\fp}\subset \A_{K,\text{fin}}$.  The reciprocity
map of class field theory provides an identification $\Gal(H_c/K)\simeq \A_K^\times/(K^\times U_c)$, where
\[
 U_c=\hat \Z^\times (1+c\hat\cO_K)\C^\times\subset \A_K^\times.
\]
For an idele $\alpha=\prod_\fp\alpha_\fp$ and an ideal $\mathfrak m$, we denote by $(\alpha)_\mathfrak{m}$ the product $\prod_{\fp\mid \mathfrak m}\alpha_\fp$. Following \cite[\S 4.1]{DRZ} we define
\[
 U_c^0=\{\alpha\in U_c\colon (\alpha)_{\fN^+}\in \ker(\psi)\subset (\Z/N^+\Z)^\times\},
\]
\[
 \bar U_c^0=\{\alpha\in  U_c\colon (\alpha)_{\bar\fN^+}\in \ker(\psi)\subset (\Z/N^+\Z)^\times\}.
\]
Here we are using the fact that $\fN^+$ has norm $N^+$, so that we have isomorphisms 
\[
\cO_{K,\fN^+}^\times/(1+
{\fN^+} \cO_{K,\fN^+})\simeq (\Z/N^+\Z)^\times,
\]
\[
\cO_{K,\bar\fN^+}^\times/(1+
{\bar\fN^+} \cO_{K,\bar\fN^+})\simeq (\Z/N^+\Z)^\times,
\]
where the notation $\cO_{K,\fN^+}$ stands for $\prod_{\fp\mid \fN^+}\cO_{K,\fp}$ and similarly for $\cO_{K,\bar\fN^+}$.

 Let $L_c$ and $L_c'$ be the fields corresponding by class field theory to $U_c^0$ and $\bar U_c^0$ respectively. That is to
say
\begin{equation}\label{eq: def of L_c}
 \Gal(L_c/K)\simeq \A_K^\times/(K^\times U_c^0), \ \ \Gal(L_c'/K)\simeq \A_K^\times/(K^\times \bar U_c^0).
\end{equation}
Both $L_c$ and $L'_c$ are quadratic extensions of $H_c$, and we denote by $\tilde H_c$ the biquadratic extension of $K$
given by $\tilde H_c=L_c L_c'$.

\begin{lemma}\label{lemma: L is contained in L_c}
 If $c$ is the one given by Lemma \ref{lemma: new}, then $L$ is contained in $L_c$.
Therefore $LL'$ is contained in $\tilde H_c$.
\end{lemma}
\begin{proof}
By class field theory it is enough to show that $U_c^0$ is contained in $\ker\chi_L$. Recall that the conductor of $\chi_L$ is equal to $\mathfrak{d}_{L/K}$ and hence equal to $c \fN_{\psi,L}$, with $\fN_{\psi,L}\mid \fN^+$. This means that $\chi_{L|\widehat\cO_K^\times}$ factors through a character
\begin{align}\label{eq: factor character}
  \chi_L\colon \cO_{K,c\fN^+}^\times/(1+c \fN^+\cO_{K,c \fN^+}) \lra \{\pm 1\}.
\end{align}
Let $(\alpha)$ be a finite idele of $K$ that belongs to $U_c^0$. We aim to see that $\chi_L(\alpha)=1$. Since $\alpha$ belongs to $U_c$, we can write it as $\alpha = a (1+cx)$ for some $a\in \hat \Z^\times$ and some $x\in\hat\cO_K$. Locally, we can express this as
\begin{align*}
  \alpha = a(1+cx)= a\prod_{\fp\nmid c\fN^+}x_\fp \prod_{\fp\mid c} (1+\fp^{v_\fp(c)}x_\fp )\prod_{\fp\mid\fN^+}x_\fp.
\end{align*}
By \eqref{eq: factor character} we see that \[\chi_L\Big(\prod_{\fp\nmid c \fN^+}x_\fp \prod_{\fp\mid c} (1+\fp^{v_\fp(c)}x_\fp ) \Big)=1.\]
Therefore, we see that 
\begin{align*}
  \chi_L(\alpha)= \chi_L \Big(a \prod_{\fp\mid\fN^+}x_\fp \Big)= \chi_L\Big( \prod_{\fp\mid\fN^+}a_\fp x_\fp \Big)\chi_L\Big( \prod_{\fp\mid c}a_\fp \Big).
\end{align*}
Since $\prod_{\fp\mid c}a_\fp$ lies in $\A_\Q^\times$, we have that 
\begin{align*}
  \chi_L\Big( \prod_{\fp\mid\fN^+}a_\fp x_\fp \Big)\chi_L\Big( \prod_{\fp\mid c}a_\fp \Big) =& 
  \psi\Big( \prod_{\fp\mid\fN^+}a_\fp x_\fp \Big)\psi\Big( \prod_{\fp\mid c}a_\fp \Big)\\
= &\psi\big( (\alpha)_{\fN^+} \big)\psi\Big( \prod_{\fp\mid c}a_\fp\Big) = \psi\Big( \prod_{\fp\mid c}a_\fp\Big).
\end{align*}
But $\prod_{\fp\mid c}a_\fp\in \prod_{\fp\mid c}\Z_{\fp_p}^\times$ (where $\fp_p=\fp\cap\Z$). Since the conductor of $\psi$ is $N_\psi$, which is coprime to $c$, we have that
\begin{align*}
  \psi\Big( \prod_{\fp\mid c}a_\fp\Big) = 1,
\end{align*}
as we aimed to show.
\end{proof}

\section{CM points on Shimura curves with quadratic character}\label{section: CM points on Shimura curves with
quadratic character}
In this section we recall some basic facts and well-known properties of Shimura curves. We also introduce the CM points that will play a key role in our construction of points in $E(M)$ later in Section \ref{section: Heegner points on Q-curves}, and we use Shimura's reciprocity law to deduce their field of definition. 

Let $\B/\Q$ be the
quaternion algebra of discriminant $N^-$. Thanks to Proposition \ref{prop: primes dividing N minus} we see that $\B$ is indefinite so we can, and do, fix an isomorphism \[\iota_\infty\colon \B\otimes_\Q
\R\stackrel{\simeq}{\lra} \M_2(\R).\] Choose
$\oR_0=\oR_0(N^+,N^-)$ an Eichler order of level $N^+$ in
$\B$
together with, for every prime $\ell\mid N^+$, an isomorphism\[\iota_\ell\colon
\B\otimes_\Q\Q_\ell\stackrel{\simeq}{\lra} \M_2(\Q_\ell)\]
such that 
\[
 \iota_\ell(\oR_0)\simeq \left\{ \smtx a b c d \in \M_2(\Z_\ell) \colon  c\in \ell \Z_\ell\right\}.
\]
In this way we also obtain an  isomorphism 
\[
\iota_{N^+}\colon \oR_0\otimes \Z_{N^+}\simeq \left\{ \smtx a b c d \in \M_2(\Z_{N^+}) \colon  c\in
N^+ \Z_{N^+}\right\}, 
\]
where $\Z_{N^+}=\prod_{p\mid N^+}\Z_p$.
 Let $\eta\colon \oR_0\ra \Z_{N^+}/N^+\Z_{N^+}$ be the map that sends $\gamma$ to
the upper left entry of $\iota_{N^+}(\gamma)$ taken modulo $N^+$. The character $\psi$ can be regarded in a natural way as a character $\psi \colon
\Z_{N^+}/N^+\Z_{N^+}\ra \{\pm 1\}.$
Let $\cU_0=\oR_0^\times$ be the group of units in $\oR_0$, and define 
\begin{equation}\label{eq: def of U_psi}
 \cU_\psi=\{\gamma\in \cU_0\colon \psi\circ\eta(\gamma)=1\}.
\end{equation}
Let  also $\Gamma_0$ (resp. $\Gamma_\psi$) denote the subgroup of norm $1$ elements in $\cU_0$ (resp. $\cU_\psi$).

\subsection{Shimura curves} Let $X_0=X_0(N^+,N^-)$ be the Shimura curve associated to $\Gamma_0$. Similarly, let  $X_\psi=X_\psi(N^+,N^-)$ be the Shimura curve
associated to $\Gamma_\psi$. See~\cite[Chapitre III]{Boutot-Carayol} for the precise moduli description. They are curves over $\Q$, whose complex points  can  be described as
\begin{equation}\label{eq: shimura curves as quotients of HH}
 X_0(\C)\simeq \Gamma_0\backslash \cH ,
\ \  X_\psi(\C)\simeq\Gamma_\psi\backslash  \cH,
\end{equation}
where $\cH$ denotes the complex upper half plane, and $\Gamma_0$ and $\Gamma_\psi$ act on $\cH$ via $\iota_\infty$. The inclusion $\Gamma_\psi\subset
\Gamma_0$ induces a degree $2$ homomorphism   defined over $\Q$ \[\pi_\psi\colon X_\psi\lra X_0.\]

\subsection{CM points}\label{subsection: CM points} Let $c$ be an integer relatively prime to $N$ and to the discriminant of $K$, and let
$\cO_c=\Z+c \cO_K$ be the order
of conductor $c$ in $K$.  An algebra embedding $\varphi\colon \cO_c\hookrightarrow \oR_0$ is said to be an \emph{optimal
embedding of conductor $c$} if \mbox{$\varphi(K)\cap \oR_0=\varphi(\cO_c)$}. Recall also the ideal $\fN^+\subset K$ of norm $N^+$ that we fixed in \S \ref{subsec: Extensions of L}, and that we denote by $\bar \fN^+$ its complex conjugate.
\begin{definition}
 We say that an optimal embedding $\varphi\colon \cO_c\hookrightarrow \oR_0$ is normalized with respect to $\fN^+$ if it
satisfies that
\begin{enumerate}
 \item $\iota_\infty (\varphi (a))\colvec \tau 1=a \colvec \tau 1$ for all $a\in \cO_c$ and all $\tau\in \C$ (here we view $K\subset\C$); and
\item $\ker(\eta\circ\varphi)=\fN^+$.
\end{enumerate}
We denote by $\cE(c,\oR_0)$ the set of normalized embeddings with respect to $\fN^+$. 
\end{definition}
\begin{remark}
  Observe that we do not impose any normalization at the primes dividing $N^-$ (in particular, the Galois action that we will introduce below on the set of normalized embeddings  will not be transitive). But the involutions at primes dividing the discriminant of the algebra do not play any role for the applications of the present note. 
\end{remark}
The groups $\Gamma_0$ and $\Gamma_\psi$ act on $\cE(c,\oR_0)$ by conjugation, and we denote by $\cE(c,\oR_0)/\Gamma_0$ and
$\cE(c,\oR_0)/\Gamma_\psi$ the corresponding (finite) sets of conjugacy classes. Each $\varphi\in \cE(c,\oR_0)$ has a unique
fixed point $\tau_\varphi$ in $\cH$. The image of $\tau_\varphi$ in $\Gamma_0\backslash\cH\simeq X_0(\C)$
(resp. in $\Gamma_\psi\backslash\cH \simeq X_\psi(\C)$) only depends on the class of $\varphi$ in $\cE(c,\oR_0)/\Gamma_0$
(resp. $\cE(c,\oR_0)/\Gamma_\psi$). We will denote the point defined by $\tau_\varphi$ in the Shimura curve by $[\tau_\varphi]$. The points obtained in this way are the so-called \emph{CM points} or \emph{Heegner points}.

We denote by $\CM_0(c)$ the set of \emph{CM points of conductor $c$} corresponding to optimal embeddings normalized with respect to $\fN^+$. That is to say
\[
 \CM_0(c)=\{ [\tau_\varphi]\in X_0(\C) \colon \varphi\in \cE(c,\oR_0)/\Gamma_0 \}.
\]
Similarly, we denote by $\CM_\psi(c)$ their preimage under $\pi_\psi$, which can be described as
\[
 \CM_\psi(c)=\{ [\tau_\varphi]\in X_\psi(\C) \colon \varphi\in \cE(c,\oR_0)/\Gamma_\psi \}.
\]
From now on we identify $\CM_0(c)$ with $\cE(c,\oR_0)/\Gamma_0$ and $\CM_\psi(c)$  with $\cE(c,\oR_0)/\Gamma_\psi$ (this is possible because the association $[\varphi]\mapsto [\tau_\varphi]$ is injective).
Every element in $\CM_0(c)$ has two preimages in $\CM_\psi(c)$, which are interchanged by the action of any element
$W_\psi\in \Gamma_0\smallsetminus\Gamma_\psi$.

There is an action $\star$ of $\hat K^\times$ on $\cE(c,\oR_0)$, given as follows. For any $x=(x_\fp)_\fp\in \hat K^\times$ and
$\varphi\in \cE(c,\oR_0)$, the fractional ideal $\hat\varphi(x)\hat\oR_0\cap \B$ is principal, say generated by
$\gamma_x\in \B^\times$. Let $a_x=\hat\varphi(x_{\fN^+})^{-1} \gamma_x$. Observe that $a_{x,\fp}\in \oR_0^\times$ for every $\fp\mid \fN^+$, and therefore it makes sense to consider $\psi\circ\eta(a_{x})$. Modifying each $\gamma_{x,\fp}$ by a unit if necessary and by strong approximation we can assume that $\gamma_x$ is chosen in such a way that $\psi\circ \eta(a_x)=1$. That is to say,
$\hat\varphi(x_{\fN^+}
)^{-1} \gamma_x$ lies in the kernel of $\psi\circ\eta$. Then $x\star
\varphi$ is defined as $x\star
\varphi:=\gamma_x^{-1}\varphi\gamma_x$.

By results of Shimura CM points are defined over $K^{ab}$, the maximal abelian extension of $K$. The Galois action on them is
given in terms of the reciprocity map of class field theory \[\operatorname{rec}\colon \hat K^\times/K^\times\lra
\Gal(K^{ab}/K)\] by means of \emph{Shimura's reciprocity law} (cf. \cite[Th. 9.6]{shimura-book}):
\begin{equation}\label{eq: shimura reciprocity law}
 \rec(x)^{-1}([\tau_\varphi])=[\tau_{x\star\varphi}].
\end{equation}
Here the action in the left is the usual Galois action on the $\Qb$-points of a variety defined over $\Q$. One of its well known consequences is that $\CM_0(c)\subset
X_0(H_c)$, i.e. CM points of conductor $c$ on $X_0$ are defined over the ring class field of conductor $c$. One can also derive from it the field of definition of $\CM_\psi(c)$, which is precisely the field $L_c$ defined in \S \ref{subsec: Extensions of L}.
\begin{proposition}\label{prop: field of def of CM points}
 $\CM_\psi(c)\subset X_\psi(L_c)$.
\end{proposition}
\begin{proof}
 It follows directly from \eqref{eq: shimura reciprocity law} and the fact that $U_c^0$ acts trivially on
$\cE(c,\oR_0)/\Gamma_\psi$.
\end{proof}

\section{ATR points on $\Q$-curves}\label{section: Heegner points on Q-curves}
In this section we introduce the main construction of this note, namely an ATR point in $E$ manufactured by means of CM points on $X_\psi$. To this end, let us briefly recall the setting of Section \ref{section: Q-curves and ATR extensions} and some of the results encountered so far. The initial data is a classical modular form $f=f_E\in S_2(N,\psi)$ such that $N$ is odd or squarefree, $\psi$ is of order $2$, and its field of Fourier coefficients $K_f$ is quadratic imaginary. Then the modular abelian variety $A_f$ is $F$ isogenous to the square of a $\Q$-curve $E$, which  is defined over the real quadratic field $F$ corresponding to $\psi$. In fact, $E/F$ is characterized by the equality of $L$-functions
\begin{align}\label{eq: L-functions}
 L(E/F,s)=\prod_{\sigma\colon K_f\hookrightarrow \C}L(f,s).
\end{align}
  Let $M$ be a quadratic ATR extension of $F$ such that $L(E/M,s)$ has sign $-1$.
This gives rise to a quadratic imaginary extension $K$, sitting in the field diagram \eqref{eq: diagram}. We assume that level $N$ factorizes as $N=N^+ N^-$, where $N^+$ is supported on the primes that split in $K$ and $N^-$ is the squarefree product of an even number of primes that are inert in $K$.

By Lemma~\ref{lemma: new} the discriminant of $L/K$ factorizes as $c \fN_{\psi,L}$, with $c\in\Z$ and
$\fN_{\psi,L}\subset K$ an integral ideal of norm $N_\psi$. Recall also that we fixed an ideal $\fN^+$ of norm $N^+$ with $\fN_{\psi,L}\mid \fN^+$. Recall also $\CM_\psi(c)$,  the set of Heegner points of conductor $c$ (and normalized with respect to $\fN^+$), which lie in $X_\psi(L_c)$.

\subsection{Construction of the ATR point}\label{subsection: construction of ATR}  Recall that, as we have seen before, by Simura's reciprocity law if $\tau\in \CM_\psi(c)$  then $\tau\in X_\psi(L_c)$. Next, we describe how to attach to each such Heegner point $\tau$ a point $P_\tau\in E(M)$.

 Let $S_2(\Gamma_\psi)=S_2(\Gamma_\psi(N^+,N^-))$ denote the space of weight two newforms with respect to $\Gamma_\psi$.
Thanks to the Jacquet--Langlands correspondence there exists a newform $g\in S_2(\Gamma_\psi)$ such that
$L(g,s)=L(f,s)$. In other words, $g$ has the same system of eigenvalues by the Hecke operators as $f$. In addition, if we let $J_\psi=\Jac(X_\psi)$ there exists a
surjective homomorphism
defined over $\Q$ (see \cite[\S 1.2.3]{Zhang_book})
\begin{align}\label{eq: pi_f}
 \pi_{f}\colon J_\psi\lra A_f.
\end{align}
The map $\pi_f$ is given explicitly in terms of integrals of $g$ and its conjugate (cf. \S \ref{subsection: complex uniformization} below). 

The next step is to associate to $\tau$ a divisor of  degree $0$ on $X_\psi$, hence a point in $J_\psi$ that we can project to $A_f$.  In the case of classical modular curves, i.e. when $N^- = 1$, the usual procedure is to use the embedding $X_\psi\hookrightarrow J_\psi$ given  by the choice of the rational cusp $\infty$ as base point. That is to say, the degree $0$ divisor attached to $\tau$ would be $(\tau)-(\infty)$. However, when $N^->1$ the Shimura curve $X_\psi$ does not have cusps. In this case, Zhang (cf. \cite[\S 1.2.2]{Zhang_book}) uses the map $\phi:X_\psi\ra J_\psi$ that sends $\tau$ to $\tau-\xi$, where $\xi\in\Pic(X_\psi)\otimes\Q$ is the so called Hodge class. Then the point $P_\tau' = \pi_f\circ\phi(\tau)$ belongs to $A_f(L_c)$.

Before continuing with the construction of the point $P_\tau$, it is worth mentioning how $P_\tau'$ (or rather a closely related point) can be computed in practice because this will be used in the explicit calculations of \S \ref{section: example}. For this we follow a remark of \cite[\S4.4]{MOK}, based on the fact that $\xi$ is a divisor of degree $1$ satisfying that  $T_\ell\xi = (\ell+1)\xi$ for all primes $\ell$ (here $T_\ell$ stands for the $\ell$-th Hecke operator). Thanks to the Hecke-equivariance of $\pi_f$, if we let $a_\ell$ denote the $\ell$-th Hecke eigenvalue of $f$ we have that
\begin{align*}
 (\ell+1-a_\ell) \pi_f(\phi(\tau)) = &\pi_f((\ell+1-T_\ell)\phi(\tau))\\ =& \pi_f((\ell+1-T_\ell)(\tau-\xi))\\ =& \pi_f((\ell+1-T_\ell)\tau).
\end{align*}
This gives an expression for $(\ell+1-a_\ell)P_\tau'$, in which we are regarding $a_\ell$ as an endomorphism of $A_f$ via the identification $K_\ell\simeq\End_\Q(A_f)\otimes\Q$.

Returning to the construction of $P_\tau$, recall that by Lemma \ref{lemma: L is contained in L_c} the field $L_c$ contains $L$.
Then we define 
\[
 P_{\tau,L}=\Tr_{L_c/L}(P_\tau')\in A_f(L).
\]
Our running assumption that the sign of the functional equation for $E/M$ is $-1$ implies that $L(E/M,1)=0$. Next, we apply the Gross--Zagier-type formula of \cite{Zhang_book} in order to show that $P_{\tau,L}$ is of infinite order in analytic rank one situations. The natural case to consider is when $L(E/F,1)\neq 0$, because otherwise the non-torsion point on $E$ would already be defined over $F$.

\begin{proposition}\label{prop: non-torsion}
 If $L(E/F,1)\neq 0$ and $L'(E/M,1)\neq 0$ then $P_{\tau,L}$ is non-torsion.
\end{proposition}
\begin{proof}
From the basic equality 
\begin{align*}
  L(E/M,s) = L(E/F,s)L(E/F,\chi_M,s)
\end{align*}
we see that $L(E/F,\chi_M,1)=0$, because of the assumption $L(E/F,1)\neq 0$. The derivative of the above expression, together with the assumption that $L'(E/M,1)\neq 0$ also implies that $L'(E/F,\chi_M,1)\neq 0$.
By \eqref{eq: L-functions}, the Artin formalism, and Lemma \ref{lemma: diagram properties} we have that
\begin{align*}
  L(E/F,\chi_M,s) =& L(f/F\otimes \chi_M,s)=L(f\otimes \Ind_F^\Q \chi_M,s )\\ =& L(f\otimes \Ind_K^\Q \chi_L,s ) = L(f/K\otimes  \chi_L,s ),
\end{align*}
from which we obtain that $L'(f/K\otimes\chi_L,1)\neq 0$.

The modular form $f$ gives rise to a cuspidal automorphic representation $\pi$ of $\GL_2(\A_\Q)$ whose central character $\omega_\pi$ is the Nebentypus $\psi$ of $f$. Since $\omega_\pi\cdot \chi_{L|\A_\Q^\times}=1$ by Lemma \ref{lemma: diagram properties}, we are in the position of applying the Gross--Zagier-type formula of \cite[\S 1.3.2]{Zhang_book}. For that, starting with the Heegner point $P_\tau'\in A_f(L_c)$ and regarding $\chi_L$ as a character of $\Gal(L_c/K)$ we set
\begin{align}\label{eq: twist of point}
  P_{\tau}^{\chi_L} := \sum_{\sigma \in\Gal(L_c/K)} \chi_L(\sigma)^{-1}\cdot \sigma(P_\tau').
\end{align}
Then \cite[Theorem 1.2]{Zhang_book} expresses the Neron--Tate heigh of $P_\tau^{\chi_L}$ as a non-zero multiple of $L'(f/K\otimes \chi_L,1)$. This implies that $P_\tau^{\chi_L}$ is non-torsion. But now, since $\chi_L$ is the character corresponding to the quadratic extension $L/K$, we can write
\begin{align*}
  P_\tau^{\chi_L} = \sum_{\sigma \in\Gal(L_c/L)} \sigma(P_\tau') + \sum_{\sigma \in\Gal(L_c/L)} (-1)\cdot (s\sigma)(P_\tau'),
\end{align*}
where $s\in\Gal(L_c/K)$ is an element that induces the nontrivial automorphism of $L/K$. This can be written as
\begin{align}\label{eq: twist of point 2}
  P_\tau^{\chi_L} = P_{\tau,L} - s(P_{\tau,L}),
\end{align}
and we see that $P_{\tau,L}$ must be non-torsion as well.
\end{proof}

Now let $\tau_M$ denote the element in $\Gal(\cM/\Q)$ whose fixed field is $M$. We define
\[
 P_{\tau,M}=P_{\tau,L} + \tau_M(P_{\tau,L}),
\]
which belongs to $A_f(M)$. We have the following consequence of Proposition \ref{prop: non-torsion}.
\begin{proposition}
 If $L(E/F,1)\neq 0$ and $L'(E/M,1)\neq 0$ then $P_{\tau,M}$ is non-torsion. 
\end{proposition}
\begin{proof}
The key property is that $P_{\tau,L}$ is defined over $L$, but it is not defined over $K$ (this follows from \eqref{eq: twist of point 2}: if it was defined over $K$, then $P_\tau^{\chi_L}$ would be 0 because $s$ would fix it, but under the hypothesis of the Proposition $P_\tau^{\chi_L}$ is non-torsion). The same is true for any multiple $n\cdot P_{\tau,L}$. Since $\tau_M(L)=L'$, we see that $\tau_M(P_{\tau,L})$ is defined over $L'$ (and is not defined over $K$). But now, if $P_{\tau,L}+\tau_M(P_{\tau,L})$ was torsion, say of order $n$, we would have that 
\begin{align*}
  nP_{\tau,L} = -n\tau_M(P_{\tau,L}),
\end{align*}
which is a contradiction because the point in the left is defined over $L$, and the point on the right is defined over $L'$.
\end{proof}
Finally, in order to define $P_\tau$, recall that $A_f$ is $F$-isogenous to $E^2$. In particular, $\Q\otimes_\Z A_f(M)\simeq \Q\otimes_\Z E(M)\times \Q\otimes_\Z E(M)$, so that we can choose a projection $\pi_E\colon A_f\ra E$ defined over $F$ such that 
\begin{align*}
  P_\tau = \pi_E(P_{\tau,M})\in E(M)
\end{align*}
is of infinite order when $P_{\tau,M}$ is.

\begin{remark}
Observe that the projection $\pi_E$ is not uniquely determined by the above condition. However, this does not affect the construction in a sensible way because the property for a point being defined over $M$ or being of infinite order is not affected by isogenies defined over $F$.
\end{remark}

One of the main motivations for the construction of the point $P_\tau$ is that it extends the construction of \cite{DRZ}
to the case  $N^->1$. However, a nice feature of the setting considered in \cite{DRZ} is that in that case the points can be
effectively computed (cf. the explicit formula of \cite[Theorem 4.6]{DRZ}) as suitable integrals of the classical
modular form $f$. In our situation, however, the equivalent computation seems to be more difficult, because the modular
forms involved are \emph{quaternionic modular forms}. This is the issue that we address in the next paragraph. As we
will see in \S\ref{section: p-adic uniformization}, the effective computation of $P_\tau$ can be accomplished by using
$p$-adic methods.

\subsection{Complex uniformization and Heegner points}  \label{subsection: complex uniformization}

 The projection map $\pi_f$ of \eqref{eq: pi_f} is given by a generalization of the
classical Eichler--Shimura construction (cf. \cite[\S4]{Da2}). In this context,
the quaternionic modular form $g$ gives rise to a differential form $\omega_g\in H^0(X_\psi,\Omega^1)$. Recall that $g$ is obtained via the Jacquet--Langlands correspondence from an elliptic modular form $f$. Denote by $f'$ the modular form whose Fourier coefficients are the complex conjugates of those of $f$, and let $g'$ denote the modular form with respect to $\Gamma_\psi$ corresponding to $f'$ by Jacquet--Langlands. Observe that $\omega_g$ and $\omega_g'$ are determined by this construction only up to multiplication by scalars, but they can be normalized so as to satisfy that $\{\omega_g,\omega_g'\}$ is a basis of the space of differential $1$-forms defined over $F$. Let $\Phi=\Phi_{N^+,N^-}$ be the map
\[
 \begin{array}{cccc}
  \Phi\colon & \operatorname{Div}^0(\cH) & \lra & \C\times \C\\
 & z_2-z_1 & \longmapsto & \left( \int_{z_1}^{z_2}\omega_g,\int_{z_1}^{z_2}\omega_{\bar g}\right).
 \end{array}
\]
The subgroup generated by the images under $\Phi$ of divisors which become trivial in
$\Gamma_\psi\backslash\cH$ is a lattice $\Lambda_{g}\subset \C\times\C$, and $\C^2/\Lambda_g$ is isogenous to $A_f(\C)$. This
gives the following analytic description of $\pi_f$:
\begin{equation}\label{eq: shimura curve parametrization}
\begin{array}{cccc}
 \Phi\colon & \operatorname{Div}^0(\cH/\Gamma_\psi)&\lra & A_f(\C)\\
 & z_2-z_1 & \longmapsto & \left( \int_{z_1}^{z_2}\omega_g,\int_{z_1}^{z_2}\omega_{\bar g}\right).
\end{array}
\end{equation}
Suppose that $D=\tau_2-\tau_1\in \Div^0 \CM_\psi(c)$. We see that the point $\pi_f(D)\in A_f(L_c)$ is given, in complex analytic terms, by the formula
\begin{align}\label{eq: pi_f(D) complex}
 \pi_f(D)= \left( \int_{\tau_1}^{\tau_2}\omega_g, \int_{\tau_1}^{\tau_2}\omega_{\bar g}\right)\in \C^2/\Lambda_g \simeq A_f(\C).
\end{align}

The effective computation of the above integrals, however, turns out to be difficult in general when $\B$ is a division algebra, because the 
newforms in $S_2(\Gamma_\psi)$ cannot be expressed as a Fourier expansion at the cusps. In the next section, and
modeling on the classical case of newforms in  $S_2(\Gamma_0)$, we will see that the points $P_\tau$ defined in \S
\ref{subsection: construction of ATR} can be computed via $p$-adic uniformization, instead of complex uniformization.

\section{$p$-adic uniformization and CM points}\label{section: p-adic uniformization}
If $p$ is a prime dividing $N^-$ the abelian varieties $J_0=\Jac(X_0)$ and  $J_\psi=\Jac(X_\psi)$ admit  rigid analytic uniformizations at $p$. That is
to say, there exist free groups of finite rank $\Lambda_0,S_0,\Lambda_\psi,S_\psi$ together with isomorphisms
\begin{equation}\label{eq: rigid analytic uniformization}
  J_0(\C_p)\simeq \Hom(S_0,\C_p^\times)/\Lambda_0, \ \ \   J_\psi(\C_p)\simeq \Hom(S_\psi,\C_p^\times)/\Lambda_\psi.
\end{equation}
In this section we use the $p$-adic uniformization of \cerednikdrinfeld{}, in the explicit formulation provided by Bertolini--Darmon, in order to give a $p$-adic analytic formula for the points $\pi_f(D)$ for $D$ a degree $0$ divisor in $J_\psi$ of \eqref{eq: pi_f}. The main feature of this formula, in contrast with that of \eqref{eq: pi_f(D) complex}, is that it is well suited for numerical computations, thanks to the explicit algorithms of \cite{FM}.
\subsection{\cerednikdrinfeld{} uniformization}\label{CD theorem} The main reference for this part is 
\cite[\S5]{Boutot-Carayol}. Recall the indefinite quaternion algebra $\B/\Q$ of discriminant $N^-$ and $\oR_0\subset\B$ the Eichler order or level $N^+$ that we fixed in \S\ref{section: CM points on Shimura curves with
quadratic character}. Now let $B/\Q$ be the definite quaternion algebra obtained from $\B$ by interchanging the invariants $p$ and $\infty$. That is to say, its set of ramification primes is
\[
 \operatorname{ram}(B)=\{ \ell \colon \ell \neq p \text{ and } \ell\mid N^-\}\cup \{\infty\}.
\]
For every $\ell\mid p N^+$ fix an isomorphism
\begin{align*}
  i_\ell\colon B\otimes\Q_\ell\lra \M_2(\Q_\ell).
\end{align*}
Let $R_0$ be a $\Z[\frac{1}{p}]$-Eichler order of level $N^+$ in $B$, which is unique up to conjugation by elements in $B^\times$. In fact, we can choose $R_0$ in such a way that is locally isomorphic to $\mathcal{R}_0$ at every prime $\ell\neq p$.  Let $\Gamma_0^{(p)}= (R_0)_1^\times$ denote the group of norm $1$ units, and let
\begin{align*}
  R_\psi = \{\gamma\in R_0\colon \gamma\in \ker (\psi\circ \eta) \},
\end{align*}
where $\eta\colon R_0\ra \Z_{N^+}/N^+\Z_{N^+}$ denotes the map that sends $\gamma$ to the upper left entry of $i_{N^+}(\gamma)$ taken modulo $N^+$. Set $\Gamma_\psi^{(p)} = (R_\psi)^\times_1$.

Both groups $\Gamma_0^{(p)}$ and  $\Gamma_\psi^{(p)}$ act on the $p$-adic upper half plane $\cH_p$ by means of $i_p$, and the
quotients $\Gamma_0^{(p)}\backslash \cH_p$ and 
$\Gamma_\psi^{(p)}\backslash \cH_p$ are rigid analytic varieties. In the following statement we collect some particular cases of the \cerednikdrinfeld{} theorem. We denote by $\Q_{p^i}$ the unramified extension of $\Q_p$ of degree $i$.
\begin{theorem}\label{th: BC} 
\begin{enumerate}
 \item $X_0\otimes\Q_{p^2}\simeq \Gamma_0^{(p)}\backslash\cH_p(\Q_{p^2})$.
\item If $p$ is split in $F$ then $X_\psi\otimes\Q_{p^2}\simeq \Gamma_\psi^{(p)}\backslash \cH_p(\Q_{p^2})$.
\item  If $p$ is inert
in $F$ then $X_\psi\otimes\Q_{p^4}\simeq \Gamma_\psi^{(p)}\backslash \cH_p(\Q_{p^4})$.
\end{enumerate}
\end{theorem}
\begin{proof}
Part $(1)$ is well know. As for parts $(2)$ and $(3)$, it follows from the \cerednikdrinfeld{}
theorem that $X_\psi\otimes\C_p\simeq \Gamma_\psi^{(p)}\backslash
\cH_p$. The only thing that we need to check is that the
isomorphism takes place after extending scalars to  $\Q_{p^2}$ if $p$ splits in $F$, and after extending scalars to
$\Q_{p^4}$ if $p$ is inert in $F$. This
follows from the discussion in
\cite[Remark 3.5.3.1]{Boutot-Carayol}. Indeed, observe that $i_p(R_\psi)\subset\M_2(\Q_p)$, contains $\smtx p 0 0 p$ if $\psi(p)=1$
(i.e., if $p$ splits in $F$), but only
contains
$\smtx{p^2}{0}{0}{p^2}$ if $\psi(p)=-1$ (i.e., if $p$ is inert in $F$). By \cite[Remark 3.5.3.1]{Boutot-Carayol} the
curve
$\Gamma_\psi^{(p)}\backslash \cH_p$ and the isomorphism to $X_\psi$ are defined over $\Q_{p^2}$ and over $\Q_{p^4}$, respectively.
\end{proof}

\subsection{Explicit $p$-adic uniformization} The main reference for this part is \cite[\S 5]{Da2}. Let $\Gamma$ be either $\Gamma_0^{(p)}$ or $\Gamma_\psi^{(p)}$.  The group $\Gamma$ acts on $\cH_p=\C_p\backslash \Q_p$ with
compact quotient. We can
speak of $S_2(\Gamma)$, the space of \emph{rigid analytic modular forms of weight $2$ on $\Gamma$}. It is
the set of
all rigid analytic functions $h\colon \cH_p\ra \C_p$ such that
\[
 h(\gamma\cdot\tau)= (c\tau+d)^2h(\tau) \text{ for all } \gamma=\smtx a b c d \in \Gamma.
\]
Let $\operatorname{Meas}_0(\P^1(\Q_p),\C_p)$ denote the space of $\C_p$-valued measures of $\P^1(\Q_p)$ with total
measure $0$. The group $\Gamma$ acts on it in the following way: if $\mu\in
\operatorname{Meas}_0(\P^1(\Q_p),\C_p)$ and $\gamma\in\Gamma$ then $ (\gamma\cdot \mu)(U)=\mu(\gamma^{-1}U)$. There is an isomorphism, due to Amice--Velu and Vishik (\cite[Corollary 2.3.4]{aws2007-dasgupta-teitelbaum})
\begin{equation}\label{eq: amice-velu-vishik}
 S_2(\Gamma)\simeq \operatorname{Meas}_0(\P^1(\Q_p),\C_p)^{\Gamma},
\end{equation}
where the superscript denotes the elements fixed by $\Gamma$.

Let $\mathcal T$ denote the Bruhat--Tits tree of $\PGL_2(\Q_p)$. Its set of vertices $\cV(\cT)$ is identified with the set of homothety $\Z_p$-lattices in $\Q_p^2$. Its set of oriented edges $\mathcal E (\mathcal T)$ consists of ordered pairs of vertices $(v_1,v_2)$ that can be represented by lattices $\Lambda_1,\Lambda_2$  such that $\Lambda_1\subset \Lambda_2$ with index $p$. For an oriented edge $e=(v_1,v_2)$, we denote $\bar e=(v_2,v_1)$, $s(e)=v_1$, and $t(e)=v_2$. An \emph{harmonic cocycle} is a function
\begin{align*}
  h\colon \cE(\cT)\lra \C_p
\end{align*}
such that $h(e)=-h(\bar e)$ for all $e\in\cE$, and such that for all $v\in\cV(\cT)$
\begin{align*}
  \sum_{s(e)=v}h(e)=0.
\end{align*}
The group $\Gamma$ acts on $\Q_p^2$ via $i_p$, and this induces an action on $\cE(\cT)$. The space of $\Gamma$-invariant measures can be identified with the set of $\Gamma$-invariant harmonic cocycles. This gives an integral structure 
\[
\operatorname{Meas}_0(\P^1(\Q_p),\Z)^{\Gamma}\subset \operatorname{Meas}_0(\P^1(\Q_p),\C_p)^{\Gamma}
\]
given by the $\Z$-valued  harmonic cocycles. It thus gives rise, via the isomorphism \eqref{eq: amice-velu-vishik}, to an integral structure $S_2(\Gamma,\Z)\subset
S_2(\Gamma)$.

If $\mu\in\operatorname{Meas}_0(\P^1(\Q_p),\Z)$ and $r$ is a continuous function on $\P^1(\Q_p)$ then the \emph{multiplicative integral} of $r$ against $\mu$ is defined as 
\begin{align*}
  \Xint\times_{\P^1(\Q_p)}r(t)d\mu(t)= \lim_{\{U_a\}} \prod r(t_a)^{\mu(U_a)},
\end{align*}
where the limit is defined over increasingly fine disjoint covers $\{U_a\}$ of $\P^1(\Q_p)$ and $t_a\in U_a$ is any sample point. 

If $h\in S_2(\Gamma,\Z)$ and $z_1,z_2\in \cH_p$ the \emph{multiplicative line integral}
 $\Xint\times_{z_1}^{z_2} h(z)dz$ is defined to be
\[
 \Xint\times_{z_1}^{z_2} h(z)dz:=\Xint\times_{\P^1(\Q)}\left(\frac{t-z_1}{t-z_2}\right) d_{\mu_h}(t), 
\]
where $\mu_h$ is the measure attached to $h$ by the isomorphism~\eqref{eq: amice-velu-vishik}. This is used to define the $p$-adic Abel--Jacobi map
\[
\begin{array}{ccccc}
 \Phi_{\text{AJ}} \colon & \Div^0(\cH_p) & \lra & \Hom(S_2(\Gamma,\Z),\C_p^\times)&\simeq (\C_p^\times)^g\\
 & z_1-z_2& \longmapsto &\left( h\mapsto \Xint\times_{z_1}^{z_2} h(z)dz \right),&
\end{array}
\]
where $g$ denotes the genus of $\Gamma\backslash \cH_p$. The group of degree $0$ divisors in $\cH_p$ that become trivial
on $\Gamma\backslash \cH_p$ are mapped by $\Phi_{\text{AJ}}$ to a lattice $\Lambda_\Gamma\subset
\Hom(S_2(\Gamma,\Z),\C_p^\times)$. This gives
\[
\begin{array}{ccccc}
 \phi_{\text{AJ}} \colon & \Div^0(\Gamma\backslash\cH_p)& \lra &
\Hom(S_2(\Gamma,\Z),\C_p^\times)/\Lambda_\Gamma\simeq \Jac(X_\Gamma)(\C_p).\\
\end{array}
\]
By particularizing this to the groups $\Gamma_0^{(p)}$ and $\Gamma_\psi^{(p)}$ one obtains an explicit expression for the rigid
analytic uniformizations of \eqref{eq: rigid analytic uniformization}:
\begin{equation*}\Div^0(\Gamma_0^{(p)}\backslash \cH_p)\simeq
J_0(\C_p)\stackrel{\Phi_{\text{AJ}}}{\lra}\Hom(S_2(\Gamma_0^{(p)},\Z),\C_p^\times)/\Lambda_0,
\end{equation*}

\begin{equation*}\Div^0(\Gamma_\psi^{(p)}\backslash \cH_p)\simeq
J_\psi(\C_p)\stackrel{\Phi_{\text{AJ}}}{\lra}\Hom(S_2(\Gamma_\psi^{(p)},\Z),\C_p^\times)/\Lambda_\psi.
\end{equation*}

\subsection{CM points and the  $p$-adic uniformization} Let $\CM_0^p(c)\subset \Gamma_0^{(p)}\backslash\cH_p$ (resp. $\CM_\psi^p(c)\subset \Gamma_\psi^{(p)}\backslash\cH_p$) denote the set of points
corresponding to $\CM_0(c)\subset X_0$ (resp. $\CM_\psi(c)\subset X_\psi$) under the isomorphism $X_0(\C_p)\simeq \Gamma_0\backslash\cH_p$ (resp. $X_\psi(\C_p)\simeq \Gamma_\psi\backslash\cH_p$). 

Bertolini and Darmon give in \cite{BD-Heegner-points} an explicit description of $\CM_0^p(c)$ in terms of certain optimal embeddings of the order of conductor $c$ into $B$. Next, we use this in order to derive the corresponding description of $\CM_\psi^p(c)$. 

Let $R_0$ be an Eichler order of $B$ of level $N^+$ as in \S\ref{CD theorem}. Let $\varphi\colon \cO_c[\frac{1}{p}]\hookrightarrow
R_0$ be an optimal embedding of $\Z[\frac{1}{p}]$-algebras. It has a single fixed point
$\tau_\varphi\in \cH_p$ satisfying
\[
 \alpha \colvec{\tau_\varphi}{1}=i_p(\varphi(\alpha))\colvec{\tau_\varphi}{1}
\]
for all $\alpha\in \cO_c[\frac{1}{p}]$. As before, we can define the notion of \emph{normalized embedding}: the
isomorphism
\[
 i_{N^+}\colon R_0\simeq \{\smtx a b c d \in \M_2(\
 \Z_{N^+})\colon c\in N^+ \Z_{N^+}\}.
\]
allows us to define the homomorphism
\[
 \eta\colon R_0\lra \Z/N^+\Z
\]
that sends each element $x$ to the upper left entry of $i_{N^+}$ modulo $N^+$. Then we say that an optimal
embedding $\varphi\colon \cO_c[\frac{1}{p}]\hookrightarrow R_0$ is normalized with respect to $\fN^+$ if
$\ker(\eta\circ\varphi)= \fN^{+}$. The explicit description of $\CM_0^p(c)$ given by Bertolini--Darmon is then:
\[
 \CM_0^p(c)=\{ [\tau_\varphi]\in \Gamma_0^{(p)}\backslash\cH_p \colon \varphi\in \cE(c,R_0)\}.
\]
Therefore, the set $\CM_\psi^p(c)$ is given by:
\[
 \CM_\psi^p(c)=\{ [\tau_\varphi]\in \Gamma_\psi^{(p)}\backslash\cH_p \colon \varphi\in \cE(c,R_0) \}.
\]
As a consequence of Proposition \ref{prop: field of def of CM points} we see that $\Phi_{\text{AJ}}(\Div^0(\CM_\psi^p(c)))$ is contained in $J_\psi(L_c)$.

\subsection{A $p$-adic analytic formula for ATR points on $\Q$-curves} Recall the modular form $f\in S_2(\Gamma_0(N),\psi)$ corresponding to $E$. There exists a rigid analytic modular form $h\in S_2(\Gamma_\psi,\C_p)$ which is an eigenvector for the Hecke operators, and has the same system of eigenvalues as $f$. Since the eigenvalues of $f$ are defined over the quadratic imaginary field $K_f$ we can identify $h$ with an harmonic cocycle with values in the ring of integers of $K_f$, and we denote by $\bar h$ the complex conjugated cocycle. Then $h_0:=(h+\bar h)/2$ and $h_1:=(h-\bar h)/2i$ belong to $S_2(\Gamma_\psi,\Z)$. Let $\Phi^{(p)}$ be the map
\begin{align*}
\begin{array}{cccc}
  \Phi^{(p)}\colon &\Div^0(\cH_p) & \lra &\C_p^\times \times \C_p^\times \\
               & z_2-z_1 & \longmapsto &{ \left( \Xint\times_{z_1}^{z_2}h_0(z)dz,\Xint\times_{z_1}^{z_2} h_1(z)dz \right)}.
\end{array}
\end{align*}
The image of the divisors whose image under $\Phi^{(p)}$ becomes trivial in $\Gamma_\psi^{(p)} \backslash\cH_p$ generates a lattice $\Lambda_f^p\subset \C_p^\times \times \C_p^\times$, and the quotient\footnote{In fact, in view of Theorem \ref{th: BC} we can even be more precise in the field of definitions and work over $\cH_p(\Q_{p^4})$; we have that $\Q_{p^4}^\times \times \Q_{p^4}^\times/\Lambda_f^p$ is isogenous to $A_f(\Q_{p^4})$} $\C_p^\times \times \C_p^\times/\Lambda_f^p$ is isogenous to $A_f(\C_p)$. In particular, if $D=\tau_2-\tau_1\in \Div^0 \CM_0^p(c)$ we find the following $p$-adic analytic formula for the corresponding CM point in $A_f$:
\begin{align}\label{eq: P_D p-adic}
 \pi_f(D)= \left(\Xint\times_{\tau_1}^{\tau_2}h_0(z)dz ,\Xint\times_{\tau_1}^{\tau_2}h_1(z)dz\right),
\end{align}
which in fact belongs to $A_f(L_c)$. The above formula for $\pi_f(D)$ can be explicitly computed, thanks to a slight modification of the explicit algorithms of~\cite{FM}. In the next section we give a detailed example on how these algorithms can be used in order to compute  in practice $\pi_f(D)$, and therefore also the point $P_\tau\in E(M)$ constructed in \S \ref{section: Heegner points on Q-curves}.

\section{An Example}\label{section: example}

The goal of this section is to illustrate with an example the construction carried out above. Let $F=\Q(\sqrt{5})$ and consider the elliptic curve defined over $F$ given as
\[
E\colon \quad y^2 = x^3 + (-432\sqrt{5}-1296)x + (-113184\sqrt{5}-282960).
\]
Remark that $E$ is a $\Q$-curve, and has conductor $39\cO_F$. We will take $p=13$. The modular form $f_E$ attached to $E$ belongs to $S_2(135,\psi)$, where $\psi$ is the unique quadratic character $\psi\colon (\Z/5\Z)^\times\to \{\pm 1\}$ of conductor $5$. Note that the form $f_E$ has field of coefficients $\Q(\sqrt{-1})$.

We need to construct the quotient of the Bruhat-Tits tree of $\GL_2(\Q_p)$ by the group $\Gamma_\psi$. In order to do so, the algorithms of~\cite{FM} have been adapted to work with congruence subgroups such as $\Gamma_\psi$. The main algorithm of~\cite{FM} returns, given two vertices or edges of the Bruhat-Tits tree, the (possibly empty) set of elements of $\Gamma_0$ relating them. One just needs to check whether the intersection of this set with $\Gamma_\psi$ is empty, which is easily done. The quotient graph that we obtain is represented in Figure~\ref{fig:fundom}. It consists of $4$ vertices and $28$ edges. The numbers next to each side of the square in Figure~\ref{fig:fundom} describe how many edges link each of the two corresponding vertices. For example, there are $8$ edges connecting $v_0$ with $v_1$. Note that all vertices have valency $14=p+1$, so all of them have trivial stabilizers.

\begin{figure}[H]
  \centering
  \includegraphics[width=.3\textwidth]{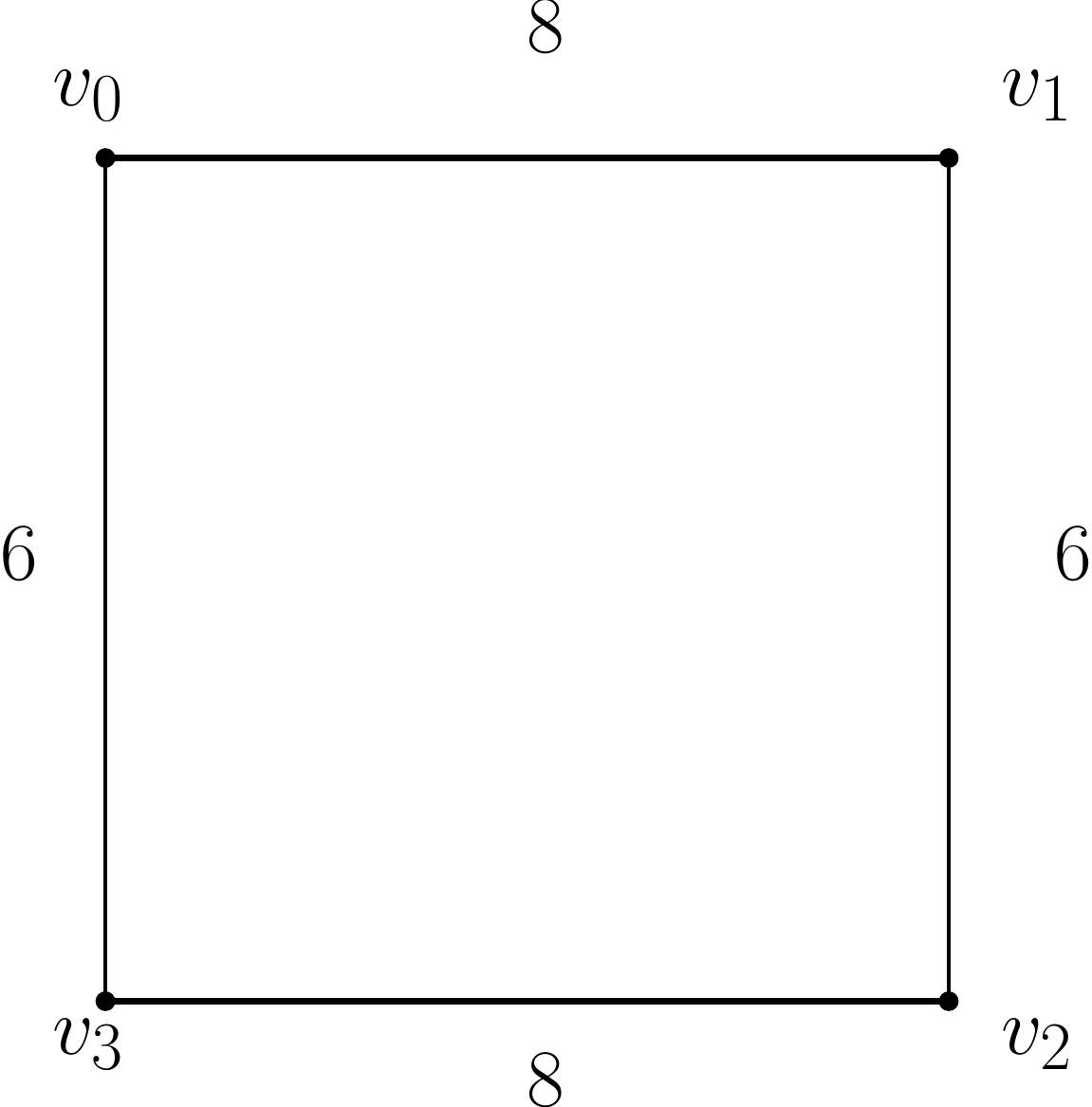}
  \caption{Quotient $\Gamma_\psi\backslash\cE(\cT)$}
  \label{fig:fundom}
\end{figure}

The space of harmonic cocycles on $\Gamma_\psi\backslash\cE(\cT)$ has dimension $25$. Taking the common eigenspace on which $T_{19}$ acts as $4$ and $T_2$ acts as $0$ we obtain a $2$-dimensional subspace associated to $f_E$. An integral basis of this subspace is given by the harmonic cocycles $h_0$ and $h_1$, which we proceed to describe. The harmonic cocycle $h_0$ has support on four edges, and takes values in $\pm 1$ there. In fact, it takes the value $+1$ and $-1$ once on two edges connecting $v_0$ and $v_3$, and the value $+1$ and $-1$ on two edges connecting $v_1$ and $v_2$. The harmonic cocycle $h_1$ can be described exactly as $h_0$, but they have disjoint supports.

Moreover, $T_3$ satisfies:
\[
T_3(h_0)=-h_1\text{ and } T_3(h_1)=h_0.
\]

Define also $\alpha = 2\sqrt{5}-1$, and let $M=F(\sqrt{\alpha})$, which is ATR. In this case, the resulting field is $K=\Q(\sqrt{-19})$, which has class number $1$. Let $g$ be a root in $\C_p$ of the polynomial $x^2 - x + 5$, and let
\[
\tau = (6   g + 1) + (8   g + 12)   13 + (7   g + 11)   13^{2} + (3   g + 3)   13^{3} + (12   g + 9)   13^{4} + (6   g + 1)   13^{5} + \cdots 
\]
be a fixed point under an embedding $\varphi$ of the maximal order of $K$ into the Eichler order $R_0(1)$ of the quaternion algebra $B= (-3, -1)$, having basis:
\[
R_0(1) = \langle 1, j, 5/2j + 5/2k, 1/2 + 1/2i - 3/2j - 3/2k \rangle.
\]
We consider the divisor $D=(\tau)-(\bar\tau)$ and calculate:
\begin{align*}
J_0 = \Xint\times_\tau^{\overline{\tau}} \omega_{h_0} = &(8   g + 12) + (3   g + 1)   13 + (7   g + 10)   13^{2} + (8   g + 8)   13^{3} + (7   g + 1)   13^{4} +\\
&(7   g + 6)   13^{5} + (9   g + 8)   13^{6} +  (7   g + 7)   13^{7} + (4   g + 9)   13^{8} + (4   g + 4)   13^{9}  +\\
&(5   g + 12)   13^{10} + (8   g + 1)13^{11} + (11   g + 11)   13^{12} +\cdots
\end{align*}
and in fact $J_1 = J_0$.

We calculate the image of $J_0$ under the Tate uniformization map, to get coordinates $(x,y)\in E(\C_p)$:
\begin{align*}
x = &(12  h^{3} + 3 h^{2} + 4  h + 1) + (9   h^{3} + 10   h^{2} + h + 9)   13 + (6   h^{3} + 5   h^{2} + 3   h + 9)   13^{2} +\\
&(6   h^{3} + 8   h^{2} + 8)   13^{3} + (8   h^{3} + 2   h^{2} + 5   h + 8)   13^{4} + (4   h^{3} + 9   h^{2} + 4   h + 6)   13^{5} + \cdots
\end{align*}
and
\begin{align*}
y = &(11   h^{3} + 5   h^{2} + 2   h + 9) + (12   h^{3} + 12   h^{2} + h + 10)   13 + (7   h^{2} + 10   h + 7)   13^{2} + \\
&(2   h^{3} + 5   h^{2} + 9   h + 7)   13^{3} + (5   h^{3} + 2   h^{2} + 4   h + 4)   13^{4} + (3   h^{3} + 3   h + 11)   13^{5} + \cdots
\end{align*}
Here, $h$ satisfies:
\[
h^4 + 3h^2 + 12h + 2 = 0.
\]

We have carried out all the calculations to precision $13^{80}$, and up to this precision it turns out that $x$ is a root of the irreducible polynomial:
\[
P_x(T) = T^{4} + 60T^{3} + 19728T^{2} + 380160T + 40144896
\]
and $y$ is a root of the irreducible polynomial:
\begin{align*}
P_y(T) = &T^{8} - 1166400T^{6} + 5027006707200T^{4} - 321342050396160000T^{2}+ \\
&75899706935371407360000.
\end{align*}
The polynomial $P_y(T)$ factors as two quartics over $F$. We let $\cM/F$ be the quartic extension generated by one of these two factors, and we remark that $P_y(T)$ splits completely over $\cM$, so it is actually the splitting field of $P_y(T)$. Let $\alpha$ be a root of $P_y(T)$ in $\cM$. Then the coordinates $(x,y)$ are defined over $\cM$ and correspond to the point:
\[
((1/12960\sqrt{5} - 1/4320)\alpha^2 + 3/2\sqrt{5} + 15/2,\alpha)\in E(\cM).
\]
Since $\cM$ contains the field $M$, we can compute the trace of this point down to $M$, to obtain the point of infinite order
\[
P_D = \left(\frac{474\sqrt{5}+750}{19} , \frac{20412\sqrt{5} + 19440}{361}\sqrt{\alpha}  \right)\in E(M).
\]

\bibliographystyle{amsalpha}
\bibliography{refs}

\def\cprime{$'$}
\providecommand{\bysame}{\leavevmode\hbox to3em{\hrulefill}\thinspace}
\providecommand{\MR}{\relax\ifhmode\unskip\space\fi MR }
\providecommand{\MRhref}[2]{%
  \href{http://www.ams.org/mathscinet-getitem?mr=#1}{#2}
}
\providecommand{\href}[2]{#2}
\begin{thebibliography}{BCDT01}

\bibitem[BC91]{Boutot-Carayol}
J.-F. Boutot and H.~Carayol, \emph{Uniformisation {$p$}-adique des courbes de
  {S}himura: les th\'eor\`emes de \v {C}erednik et de {D}rinfel\cprime d},
  Ast\'erisque (1991), no.~196-197, 7, 45--158 (1992), Courbes modulaires et
  courbes de Shimura (Orsay, 1987/1988). \MR{1141456 (93c:11041)}

\bibitem[BCDT01]{BCDT}
Christophe Breuil, Brian Conrad, Fred Diamond, and Richard Taylor, \emph{On the
  modularity of elliptic curves over {$\bold Q$}: wild 3-adic exercises}, J.
  Amer. Math. Soc. \textbf{14} (2001), no.~4, 843--939 (electronic).
  \MR{1839918 (2002d:11058)}

\bibitem[BD98]{BD-Heegner-points}
Massimo Bertolini and Henri Darmon, \emph{Heegner points, {$p$}-adic
  {$L$}-functions, and the {C}erednik-{D}rinfeld uniformization}, Invent. Math.
  \textbf{131} (1998), no.~3, 453--491. \MR{1614543 (99f:11080)}

\bibitem[BD09]{BD-genus}
\bysame, \emph{The rationality of {S}tark-{H}eegner points over genus fields of
  real quadratic fields}, Ann. of Math. (2) \textbf{170} (2009), no.~1,
  343--370. \MR{2521118 (2010m:11072)}

\bibitem[BDP13]{BDP}
Massimo Bertolini, Henri Darmon, and Kartik Prasanna, \emph{Generalized
  {H}eegner cycles and {$p$}-adic {R}ankin {$L$}-series}, Duke Math. J.
  \textbf{162} (2013), no.~6, 1033--1148, With an appendix by Brian Conrad.
  \MR{3053566}

\bibitem[BDR]{BDR}
Massimo Bertolini, Henri Darmon, and Victor Rotger, \emph{{B}eilinson--{F}lach
  elements and {E}uler systems {I}: syntomic regulators and $p$-adic {R}ankin
  $l$-series}, {P}reprint.

\bibitem[Dar01]{Da1}
Henri Darmon, \emph{Integration on {$\mathcal{H}_p\times\mathcal{H}$} and
  arithmetic applications}, Ann. of Math. (2) \textbf{154} (2001), no.~3,
  589--639. \MR{1884617 (2003j:11067)}

\bibitem[Dar04]{Da2}
\bysame, \emph{Rational points on modular elliptic curves}, CBMS Regional
  Conference Series in Mathematics, vol. 101, Published for the Conference
  Board of the Mathematical Sciences, Washington, DC, 2004. \MR{2020572
  (2004k:11103)}

\bibitem[DG02]{darmon-green}
Henri Darmon and Peter Green, \emph{Elliptic curves and class fields of real
  quadratic fields: algorithms and evidence}, Experiment. Math. \textbf{11}
  (2002), no.~1, 37--55. \MR{1960299 (2004c:11112)}

\bibitem[DL03]{DL}
Henri Darmon and Adam Logan, \emph{Periods of {H}ilbert modular forms and
  rational points on elliptic curves}, Int. Math. Res. Not. (2003), no.~40,
  2153--2180. \MR{1997296 (2005f:11110)}

\bibitem[DP06]{darmon-pollack}
Henri Darmon and Robert Pollack, \emph{Efficient calculation of
  {S}tark-{H}eegner points via overconvergent modular symbols}, Israel J. Math.
  \textbf{153} (2006), 319--354. \MR{2254648 (2007k:11077)}

\bibitem[DRa]{DR1}
Henri Darmon and Victor Rotger, \emph{Diagonal cycles and {E}uler systems {I}:
  a $p$-adic {G}ross--{Z}agier formula}, {A}nnales {S}c. {E}cole {N}ormal
  {S}uperieure, To appear.

\bibitem[DRb]{DR2}
\bysame, \emph{Diagonal cycles and {E}uler systems {II}: The {B}irch and
  {S}winnerton--{D}yer conjecture for {H}asse--{W}eil-{A}rtin {$L$}-functions},
  {P}reprint.

\bibitem[DRZ12]{DRZ}
Henri Darmon, Victor Rotger, and Yu~Zhao, \emph{The {B}irch and
  {S}winnerton-{D}yer conjecture for $\mathbb{Q}$-curves and {O}da's period
  relations}, 1--40.

\bibitem[DT08]{aws2007-dasgupta-teitelbaum}
Samit Dasgupta and Jeremy Teitelbaum, \emph{The $p$-adic upper half plane},
  $p$-adic geometry : lectures from the 2007 Arizona Winter School (Matthew
  Baker, Brian Conrad, Samit Dasgupta, Kiran~S. Kedlaya, and Jeremy Teitelbaum,
  eds.), 2008.

\bibitem[Elk94]{El}
Noam~D. Elkies, \emph{Heegner point computations}, Algorithmic number theory
  ({I}thaca, {NY}, 1994), Lecture Notes in Comput. Sci., vol. 877, Springer,
  Berlin, 1994, pp.~122--133. \MR{1322717 (96f:11080)}

\bibitem[Elk98]{El-SC}
\bysame, \emph{Shimura curve computations}, Algorithmic number theory
  ({P}ortland, {OR}, 1998), Lecture Notes in Comput. Sci., vol. 1423, Springer,
  Berlin, 1998, pp.~1--47. \MR{1726059 (2001a:11099)}

\bibitem[FM14]{FM}
Cameron Franc and Marc Masdeu, \emph{Computing fundamental domains for the
  {B}ruhat-{T}its tree for {${\rm GL}_2(\bold{Q}_p)$}, {$p$}-adic automorphic
  forms, and the canonical embedding of {S}himura curves}, LMS J. Comput. Math.
  \textbf{17} (2014), no.~1, 1--23. \MR{3230854}

\bibitem[G{\"a}r12]{Ga-art}
J{\'e}r{\^o}me G{\"a}rtner, \emph{Darmon's points and quaternionic {S}himura
  varieties}, Canad. J. Math. \textbf{64} (2012), no.~6, 1248--1288.
  \MR{2994664}

\bibitem[GJG10]{GoGu}
Enrique Gonz{\'a}lez-Jim{\'e}nez and Xavier Guitart, \emph{On the modularity
  level of modular abelian varieties over number fields}, J. Number Theory
  \textbf{130} (2010), no.~7, 1560--1570. \MR{2645237 (2011f:11073)}

\bibitem[GL01]{Pep-Lario}
Josep Gonz{\'a}lez and Joan-C. Lario, \emph{{$\bold Q$}-curves and their
  {M}anin ideals}, Amer. J. Math. \textbf{123} (2001), no.~3, 475--503.
  \MR{1833149 (2002e:11070)}

\bibitem[GM]{GM-elt}
X.~{Guitart} and M.~{Masdeu}, \emph{{Elementary Matrix Decomposition and The
  Computation of Darmon Points with Higher Conductor}}, {M}ath. {C}omp., {T}o
  appear.

\bibitem[GM14]{GM-quat}
\bysame, \emph{{{O}verconvergent cohomology and quaternionic {D}armon points}},
  J. London Math. Soc. (2014).

\bibitem[Gre06]{Gr2}
Matthew Greenberg, \emph{Heegner point computations via numerical {$p$}-adic
  integration}, Algorithmic number theory, Lecture Notes in Comput. Sci., vol.
  4076, Springer, Berlin, 2006, pp.~361--376. \MR{2282936 (2008a:11069)}

\bibitem[Gre09]{Gr}
\bysame, \emph{{S}tark-{H}eegner points and the cohomology of quaternionic
  {S}himura varieties}, Duke Math. J. \textbf{147} (2009), no.~3, 541--575.
  \MR{2510743 (2010f:11097)}

\bibitem[Gro87]{Gross}
Benedict~H. Gross, \emph{Heights and the special values of {$L$}-series},
  Number theory ({M}ontreal, {Q}ue., 1985), CMS Conf. Proc., vol.~7, Amer.
  Math. Soc., Providence, RI, 1987, pp.~115--187. \MR{894322 (89c:11082)}

\bibitem[GSS]{GrSeSh}
Matthew Greenberg, Marco~Adamo Seveso, and Shahab Shahabi, \emph{Modular
  $p$-adic $l$-functions attached to real quadratic fields and arithmetic
  applications}, To appear in Crelle's Journal.

\bibitem[GZ86]{GZ}
Benedict~H. Gross and Don~B. Zagier, \emph{Heegner points and derivatives of
  {$L$}-series}, Invent. Math. \textbf{84} (1986), no.~2, 225--320. \MR{833192
  (87j:11057)}

\bibitem[Kol88]{Ko}
V.~A. Kolyvagin, \emph{Finiteness of {$E({\bf Q})$} and {SH{$(E,{\bf Q})$}} for
  a subclass of {W}eil curves}, Izv. Akad. Nauk SSSR Ser. Mat. \textbf{52}
  (1988), no.~3, 522--540, 670--671. \MR{954295 (89m:11056)}

\bibitem[KW09a]{KW1}
Chandrashekhar {Khare} and Jean-Pierre {Wintenberger}, \emph{{Serre's
  modularity conjecture. I.}}, {Invent. Math.} \textbf{178} (2009), no.~3,
  485--504 (English).

\bibitem[KW09b]{KW2}
\bysame, \emph{{Serre's modularity conjecture. II.}}, {Invent. Math.}
  \textbf{178} (2009), no.~3, 505--586 (English).

\bibitem[LV14]{LV}
Matteo Longo and Stefano Vigni, \emph{The rationality of quaternionic {D}armon
  points over genus fields of real quadratic fields}, Int. Math. Res. Not. IMRN
  (2014), no.~13, 3632--3691. \MR{3229764}

\bibitem[Mok11]{MOK}
Chung~Pang Mok, \emph{Heegner points and {$p$}-adic {$L$}-functions for
  elliptic curves over certain totally real fields}, Comment. Math. Helv.
  \textbf{86} (2011), no.~4, 867--945. \MR{2851872}

\bibitem[{Nel}12]{nelson}
P.~D. {Nelson}, \emph{{Evaluating modular forms on Shimura curves}}, ArXiv
  e-prints (2012).

\bibitem[PS11]{PS}
Robert Pollack and Glenn Stevens, \emph{Overconvergent modular symbols and
  {$p$}-adic {$L$}-functions}, Ann. Sci. \'Ec. Norm. Sup\'er. (4) \textbf{44}
  (2011), no.~1, 1--42. \MR{2760194 (2012m:11074)}

\bibitem[Pyl04]{Pyle}
Elisabeth~E. Pyle, \emph{Abelian varieties over {$\Bbb Q$} with large
  endomorphism algebras and their simple components over {$\overline{\Bbb
  Q}$}}, Modular curves and abelian varieties, Progr. Math., vol. 224,
  Birkh\"auser, Basel, 2004, pp.~189--239. \MR{2058652 (2005f:11119)}

\bibitem[{Que}12]{Quer}
J.~{Quer}, \emph{{Package description and tables for the paper ''Fields of
  definition of building blocks''}}, ArXiv e-prints (2012).

\bibitem[Rib77]{RibetGR}
Kenneth~A. Ribet, \emph{Galois representations attached to eigenforms with
  {N}ebentypus}, Modular functions of one variable, {V} ({P}roc. {S}econd
  {I}nternat. {C}onf., {U}niv. {B}onn, {B}onn, 1976), Springer, Berlin, 1977,
  pp.~17--51. Lecture Notes in Math., Vol. 601. \MR{0453647 (56 \#11907)}

\bibitem[Rib80]{RibetTwists}
\bysame, \emph{Twists of modular forms and endomorphisms of abelian varieties},
  Math. Ann. \textbf{253} (1980), no.~1, 43--62. \MR{594532 (82e:10043)}

\bibitem[Rib04]{Ribet}
Kenneth~A. Ribet, \emph{Abelian varieties over $\q$ and modular forms},
  Progress in Math. 224, Birkh\"ausser, 2004, pp.~241--261.

\bibitem[Rot08]{Ro}
Victor Rotger, \emph{Which quaternion algebras act on a modular abelian
  variety?}, Math. Res. Lett. \textbf{15} (2008), no.~2, 251--263. \MR{2385638
  (2009j:11099)}

\bibitem[Shi94]{shimura-book}
Goro Shimura, \emph{Introduction to the arithmetic theory of automorphic
  functions}, Publications of the Mathematical Society of Japan, vol.~11,
  Princeton University Press, Princeton, NJ, 1994, Reprint of the 1971
  original, Kan{\^o} Memorial Lectures, 1. \MR{1291394 (95e:11048)}

\bibitem[TW95]{TW}
Richard Taylor and Andrew Wiles, \emph{Ring-theoretic properties of certain
  {H}ecke algebras}, Ann. of Math. (2) \textbf{141} (1995), no.~3, 553--572.
  \MR{1333036 (96d:11072)}

\bibitem[VW14]{VW}
John Voight and John Willis, \emph{Computing power series expansions of modular
  forms}, Computations with modular forms, Contrib. Math. Comput. Sci., vol.~6,
  Springer, Berlin, 2014, pp.~331--361.

\bibitem[Was97]{washington-book}
Lawrence~C. Washington, \emph{Introduction to cyclotomic fields}, second ed.,
  Graduate Texts in Mathematics, vol.~83, Springer-Verlag, New York, 1997.
  \MR{1421575 (97h:11130)}

\bibitem[Wil95]{Wi}
Andrew~J. Wiles, \emph{Modular elliptic curves and {F}ermat's last theorem},
  Ann. of Math. (2) \textbf{141} (1995), no.~3, 443--551. \MR{1333035
  (96d:11071)}

\bibitem[YZZ12]{Zhang_book}
Xinyi Yuan, Shou-Wu Zhang, and Wei Zhang, \emph{The gross-zagier formula on
  shimura curves}, Annals of Mathematics Studies, vol.~84, Princeton University
  Press, 2012.

\end{thebibliography}

\end{document}